\documentclass{amsart}
\usepackage{yhmath}  
\usepackage{amsmath,amsthm}
\usepackage{marvosym}
\usepackage{graphicx}
\usepackage{color}
\usepackage{epsfig}
\usepackage{morefloats}
\usepackage{hyperref}
\usepackage{fullpage}
\usepackage[active]{srcltx}
\theoremstyle{plain}
\newtheorem{thm}{Theorem}[section]
\newtheorem{cor}[thm]{Corollary}
\newtheorem{conj}[thm]{Conjecture}
\newtheorem{prop}[thm]{Proposition}
\newtheorem{lemma}[thm]{Lemma}

\newtheorem*{claim*}{Claim}

\theoremstyle{definition}

\newtheorem{remark}[thm]{Remark}
\newtheorem{example}[thm]{Example}

\newcommand{\comment}[1]{}

\newcommand{\bdry}{\ensuremath{\partial}}

\newcommand{\N}{\ensuremath{\mathbb{N}}}
\newcommand{\Q}{\ensuremath{\mathbb{Q}}}
\newcommand{\R}{\ensuremath{\mathbb{R}}}
\newcommand{\Z}{\ensuremath{\mathbb{Z}}}
\newcommand{\C}{\ensuremath{\mathbb{C}}}

\newcommand{\0}{{\mbox{\sc{O}}}}
\newcommand{\I}{{\mbox{\sc{I}}}}
\newcommand{\II}{{\mbox{\sc{I\!I}}}}
\newcommand{\III}{{\mbox{\sc I\!I\!I}}}
\newcommand{\IV}{{\mbox{\sc{I\!V}}}}
\newcommand{\V}{{\mbox{\sc{V}}}}
\newcommand{\VI}{{\mbox{\sc{V\!I}}}}

\newcommand{\mobius}{M\"{o}bius }

\newcommand{\cut}{\ensuremath{\backslash}}

\definecolor{light-gray}{gray}{0.75}

\makeatletter
\newcommand\textsubscript[1]{\@textsubscript{\selectfont#1}}
\def\@textsubscript#1{{\m@th\ensuremath{_{\mbox{\fontsize\sf@size\z@#1}}}}}
\newcommand\textbothscript[2]{%
  \@textbothscript{\selectfont#1}{\selectfont#2}}
\def\@textbothscript#1#2{%
  {\m@th\ensuremath{%
    ^{\mbox{\fontsize\sf@size\z@#1}}%
    _{\mbox{\fontsize\sf@size\z@#2}}}}}
\def\@super{^}\def\@sub{_}

\catcode`^\active\catcode`_\active
\def\@super@sub#1_#2{\textbothscript{#1}{#2}}
\def\@sub@super#1^#2{\textbothscript{#2}{#1}}
\def\@@super#1{\@ifnextchar_{\@super@sub{#1}}{\textsuperscript{#1}}}
\def\@@sub#1{\@ifnextchar^{\@sub@super{#1}}{\textsubscript{#1}}}
\def^{\let\@next\relax\ifmmode\@super\else\let\@next\@@super\fi\@next}
\def_{\let\@next\relax\ifmmode\@sub\else\let\@next\@@sub\fi\@next}
\makeatother

\title{
The Classification of Rational Subtangle Replacements between Rational Tangles}

\author{Kenneth L.\ Baker}
\address{
Department of Mathematics,
University of Miami,
PO Box 249085
Coral Gables, FL 33124-4250}
\email{k.baker@math.miami.edu}
\urladdr{http://math.miami.edu/\char126 kenken}

\author{Dorothy Buck}
\address{Dept of Mathematics, Imperial College London, South Kensington, London England SW7 2AZ}
\email{d.buck@imperial.ac.uk}
\urladdr{http://www2.imperial.ac.uk/~dbuck/}

\begin{document}

\begin{abstract}
A natural generalization of a crossing change is a rational subtangle replacement (RSR).
We characterize the fundamental situation of the rational tangles obtained from a given rational tangle via RSR, building on work of Berge and Gabai, and determine the sites where these RSR may occur.
In addition we also determine the sites for RSR distance at least two between 2-bridge links. 
These proofs depend on the geometry of the branched double cover.
Furthermore, we classify all knots in lens spaces whose exteriors are generalized Seifert fibered spaces and their lens space surgeries, extending work of Darcy-Sumners.   
This work is in part motivated by the common biological situation of proteins cutting, rearranging and resealing DNA segments -- effectively performing RSR on DNA `tangles'.
\end{abstract}

\maketitle

\section{Introduction}

All knots in the $3$--sphere are related by a sequence of crossing changes. This leads to deep questions:  Can a single crossing change transform a given knot into the unknot?  \textit{Which} two links related by a single crossing change? When do two different crossing changes relate the same pair of links? \textit{Where} are all these crossing changes?

A \textit{rational (sub)tangle replacement}, RSR, is a natural generalization of a crossing change, and thus prompts similar questions.    
An RSR is the excision of a rational subtangle from a $3$--manifold with a properly embedded $1$--manifold followed by the insertion of another.   Foundational and recent articles addressing these questions, using a variety of techniques, include \cite{taylor,EudaveMunoz,Lick2,ernst,ES1,withMauro,withCyn1,withCyn2,Mariel,darcysumners,DarcyLuecke,Wu}.

The primary goal of this article is to address these questions for the foundational case of RSR within rational tangles.  We classify both \textit{which} pairs of rational tangles may be related by an RSR and \textit{where} these RSR occur.\footnote{ We use the term {\em rational subtangle replacement} instead of the more common {\em rational tangle replacement} to emphasize that the replacement itself occurs within a rational tangle and is generically not the replacement of that entire rational tangle. }

The \textit{which} part of this classification follows fairly directly from work of Berge \cite{berge} and Gabai \cite{gabai} (and Moser \cite{moser}) on surgeries on knots in solid tori yielding solid tori via the Montesinos Trick \cite{montesinos}.  We exhibit this more explicitly by presenting the corresponding rational tangles and an RSR between them.  While Berge classifies the knots in solid tori with solid torus surgeries, the \textit{where} part of the classification of RSR between rational tangles does not follow from the Montesinos trick.  Indeed two non-homeomorphic tangles may have homeomorphic branched double covers.  We complete the \textit{where} part of this classification in part by generalizing a theorem of Ernst \cite{ernst}, in part by adapting work of Paoluzzi \cite{paoluzzi}, and in part by regarding the tangles as hyperbolic orbifolds.  The first relies upon the corresponding knot exteriors being Seifert fibered, the second addresses mutations of tangles and involutions of manifolds with non-trivial JSJ decompositions, while the last relies upon the hyperbolic orbifold surgery theorem.

We also discuss the related questions of which pairs of $2$--bridge links are related by an RSR and where these RSR occur.  We revisit the classification by Darcy-Sumners \cite{darcysumners} (see also Torisu \cite{torisu2}) of which $2$--bridge links are related by an RSR of distance at least $2$ (precise definitions will be given in Subsection~\ref{sec:rsrdef})  and apply Ernst's theorem to classify where these RSR occur.  Work of Greene \cite{greene} gives the classification of which $2$--bridge links are related to the unknot by a distance $1$ RSR while work of Lisca \cite{lisca} can be extended to give the classification of which $2$--bridge links are related to the two component unlink by a distance $1$ RSR\footnote{Greene \cite{greene} has a remark that suggests Rasmussen had observed this result follows from Lisca's earlier work \cite{lisca}.}.  In both these cases, there are conjectural pictures of where such RSR occur.
The general question of which $2$--bridge links are related by a distance $1$ RSR is still open.

This project is in part motivated by the common biological situation of proteins cutting, rearranging and resealing DNA segments -- effectively performing RSR on DNA `tangles'.

\subsection{RSR between rational tangles}
As rational tangles are commonly parametrized by the extended rational numbers, we provide the classification of which rational tangles are related by an RSR in terms of this parametrization.
Our main result is the following:

\begin{thm}\label{thm:full}
If an RSR of distance $d$ takes the $p/q$--tangle to the $u/v$--tangle then $u/v$ belongs to one of the following families of rational numbers (depending on $p/q$ and $d$):

\begin{itemize}
\item[\0.] $(d\geq1) \quad \Big\{ \frac{u}{v} \, \Big\vert \, pv-qu =  \pm d \Big\}$  and the RSR is the full replacement,

\item[\I.]  $(d\geq1) \quad  \Big\{ \frac{p +\epsilon d a(aq-bp)}{q +\epsilon db(aq-bp)}  \, \Big\vert \, a,b \mbox{ coprime}, \epsilon = \pm  \Big\}$,

\item[\II.] $(d=1) \quad \Big\{  \frac{p +\epsilon 4 a(aq-bp)}{q +\epsilon 4 b(aq-bp)}  \, \Big\vert \,  a,b \mbox{ coprime}, \epsilon = \pm  \Big\}$,

\item[\III.] $(d=1)  \quad \Big\{ \frac{p(b-1)(4ab-4a-2b-1) +\epsilon p'(2ab-2a-b)^2 }{q(b-1)(4ab-4a-2b-1) +\epsilon q'(2ab-2a-b)^2}, 
\frac{p(1-2a)^2(b-1) +\epsilon p'(2ab-2a-b)^2 }{q(1-2a)^2(b-1) +\epsilon q'(2ab-2a-b)^2} \, \Big\vert $
\begin{flushright}
$ a,b,p',q' \in \Z, \, pq'-p'q=1, \epsilon = \pm \Big\}$, 
\quad \quad \quad
\end{flushright}

\item[\IV.] $(d=1)$
$\Big\{ \frac{p(2a-1)(2ab+a-b+1) +\epsilon p'(2ab+a-b)^2}{q(2a-1)(2ab+a-b+1) +\epsilon q'(2ab+a-b)^2},
 \frac{p(2b+1)(2ab+a-b+1)  +\epsilon p'(2ab+a-b)^2}{q(2b+1)(2ab+a-b+1)  +\epsilon q'(2ab+a-b)^2} \, \Big\vert$
 \begin{flushright}
$ a,b,p',q' \in \Z, \, pq'-p'q=1, \epsilon = \pm \Big\}$.
\quad \quad \quad 
\end{flushright}
\end{itemize}

Moreover the RSR occur, up to homeomorphism, only at the core arc for family \0\ or as indicated in Figure~\ref{fig:d=1} 
for families \I\ -- \IV.
\end{thm}

\begin{proof}
Apply Lemma~\ref{lem:pairsoftangles} to Theorem~\ref{thm:mainthm}, sending $1/0$ to $p/q$ and $r/s$ to $u/v$.  

In doing so,  note that for family \I\  we may take $s' = 1-dab$ since $(1-dab)(1+dab) \equiv 1 \mod b^2$. Therefore Lemma~\ref{lem:pairsoftangles} sends $\frac{r}{s} = \frac{1+dab}{da^2}$ to 
\[\frac{u}{v} \in \left\{ \frac{p(1+\delta dab)+\epsilon p' (da^2)}{q(1+\delta dab)+\epsilon q' (da^2)}\  \Big\vert \  p', q' \in \Z, p q' - q p'=1, \epsilon = \pm 1, \delta = \pm1 \right\}.\] 
For each $p', q' \in \Z$ with $p q' - q p'=1$ replace the coprime pair $(a,b)$ with the coprime pair $(qA-pB, \epsilon \delta(q'A-p'B))$.  This leads to the form stated for family \I.     Family \II\ works similarly, but with $d$ replaced by $4$.

For families \III\ and \IV, we may take $s'$ to be the denominator of  $[0,-a,-2,-b,-2, a, 1]= \frac{4ab-4a-2b+3}{(1-2a)^2(b-1)}$ and $[0,-b-1, -1, 1, -a +1, b, -2]=\frac{4ab+2a-2b+3}{(2b+1)(2ab+a-b+1)}$ respectively.

Theorem~\ref{thm:sitessimple} gives the last statement of this theorem.
\end{proof}

Our results may find use when viewed in the context of cobordisms of spheres with four marked points.   Our classification determines the tangles in $S^2 \times I$ for which one end may be capped off with a rational tangle in at least two different ways to produce rational tangles.

\begin{thm}
Up to homeomorphism, the tangles in $S^2 \times I$ for which one boundary component may be filled in at least two different ways with a rational tangle to yield a rational tangle are those shown in Figure~\ref{fig:fillingsofS2xI}.   In particular:
\begin{itemize}
\item[$\Q$.] If every rational tangle filling produces a rational tangle, then the tangle is the product tangle.
\item[$\Z$.] If a $1$--parameter family of fillings produces rational tangles, then the tangle belongs to family \I.
\item[3.] If exactly three fillings produce rational tangles, then the tangle is the Berge Tangle.
\item[2.] If exactly two fillings produce rational tangles, then the tangle belongs to family \II, \III, or \IV.
\end{itemize}
\end{thm}

\begin{proof}
This follows from Theorem~\ref{thm:sitessimple}, the classification of sites upon which an RSR between rational tangles may occur.  The exteriors of these sites are then the tangles in $S^2 \times I$ of the present theorem.  Section~\ref{sec:bergetangle} addresses the division according to the number of fillings.  Figure~\ref{fig:fillingsofS2xI} shows the various tangles in $S^2 \times I$ and their rational tangle fillings to their left.
\end{proof}

\subsection{Biological Motivation}

Many proteins operate on DNA by cutting, rearranging and resealing the DNA molecule in a localised way. This cut-and-seal mechanism can be simple -- e.g.\  type II topoisomerases preferentially unknot or unlink DNA molecules by performing crossing changes on the DNA axis \cite{ZechNAR,VologNAR}.  Or this mechanism can be quite complex -- e.g. site-specific recombinases invert, excise and insert DNA segments by concerted strand exchanges, crossover reactions and ligations that can be intricately choreographed \cite{GrindRice}. 

A natural way to model this is to represent the axis of the DNA double helix as a curve, and to capture the (often intertwined) geometry of this axis as various \textit{tangles} (defined in Section \ref{sec:defn}).  For example, the topoisomerase-mediated crossing change can be modelled as converting a $(+1)$ tangle into a $(-1)$ tangle or \textit{vice versa}.  

When the DNA molecule is covalently closed, these localized DNA transformations can be modelled in terms of the initial and resultant DNA knots or links.  This approach, pioneered by Ernst and Sumners \cite{ES1} and now expanded and employed by a number of researchers (\cite{review} and references therein),  has enjoyed tremendous success.   Here, several tangles -- each representing particular segments of the DNA axis -- are glued together in particular ways to form a specific knot or link.  The protein action is modelled as pulling out one particular tangle and replacing it with another, thus converting one specific knot into another.  The general strategy is then to use these known knots as probes to determine the constituent tangles by studying the Dehn surgeries between the branched double covers of these knots. 
The predictions of these models have been used to illuminate both the mechanism and the pathway of these protein-mediated DNA reactions, e.g.  \cite{withIan}.

However, an outstanding issue has been what to do when the DNA axis is linear, or in fact a knot or link that is not a 4-plat\footnote{For example, the most typical knots arising from site-specific recombination are small Montesinos knots and links \cite{withErica1,withErica2}, a superfamily of 4-plats.}.  Here the typical arguments begin to falter; in particular, one can no longer rely on the Cyclic Surgery Theorem \cite{cgls} to classify lens space surgeries.  The current work --- by focusing exclusively on the replacement of (rational) tangles --- classifies when these types of protein-DNA interactions can occur without relying on a global topology.  Thus one can now model recombination or topoisomerase simplification of these non-4-plat knots or links, or indeed, of linear DNA.  

For example, Theorem~\ref{thm:sites2bridge} can restrict the possibile mechanism of type II topoisomerases (proteins whose sole function is to unknot or unlink DNA molecules).  
Among other things, this theorem determines where (up to homeomorphism) the two DNA segments must be to perform a crossing change converting one specified $2$--bridge link into another.  In particular then, Theorem~\ref{thm:sites2bridge} characterizes the most effective locations for maximal unknotting/unlinking efficiency --- a current hotly debated question (see e.g.\ \cite{VologNAR,ZechNAR,TopoRef3,TopoRef4,TopoRef5,TopoRef6,TopoRef7,TopoRef8}).

\subsection{Overview}
Section~\ref{sec:defn} gives the basic definitions and notation that will be used throughout.  Then the proof of Theorem~\ref{thm:mainthm}, which classifies \textit{which} RSR can occur, is assembled in Section~\ref{sec:RSR}.  
Section~\ref{sec:bergetangle} exhibits a non-trivial tangle in $S^2 \times I$, the {\em Berge tangle},  that has three mutually distance $1$ rational tangle fillings.  
In Section~\ref{sec:2bridge} we discuss recent classifications of distance $1$ RSR between $2$--bridge links and either the unknot or unlink of two components.  Furthermore, to highlight the difference in RSR between rational tangles and RSR between $2$--bridge links, Section~\ref{sec:4platadj} provides an example of a pair of  $2$--bridge links related by a distance $1$ RSR that does not arise from the closures of a pair of rational tangles related by a distance $1$ RSR. 

In Section~\ref{sec:sites} we discuss the sites \textit{where} RSR can occur.   Up to homeomorphism, Theorem~\ref{thm:sitessimple} determines the sites of RSR between rational tangles while Theorem~\ref{thm:sites2bridge} determines the sites of distance $d \geq 2$ RSR between $2$--bridge links.   The techniques we use in the proofs depend on the geometry of branched double cover.  For distance 1 RSR in Theorem~\ref{thm:sitessimple} the site lifts to a knot whose exterior can be hyperbolic, a cable space, or Seifert fibered.  For the hyperbolic manifolds we treat the tangles as hyperbolic orbifolds and consider surgery for orbifolds.  

Finally, to further illuminate work contained in \cite{darcysumners}, in Section~\ref{sec:seifertfibered} we catalogue the knots in lens spaces whose exteriors are generalized Seifert fibered spaces and the surgeries on them which produce lens spaces.
Of particular note, we identify and study the knots isotopic to a regular fiber in a true Seifert fibration with non-orientable base of a lens space; among those with lens space surgeries, all may be viewed as lens space torus knots except one up to homeomorphism.

\subsection{Acknowledgements}
The authors would like to thank Dror Bar-Natan, Josh Greene, Neil Hoffman, Jesse Johnson, Louis Kauffman, Andrew Lobb, John Luecke, Nikolai Saveliev, and Mike Williams for useful conversations.  DB is partially supported by EPSRC grants EP/H0313671, EP/G0395851 and  EP/J1075308.  KB is partially supported by the University of Miami 2011 Provost Research Award and by a grant from the Simons Foundation (\#209184 to Kenneth Baker).

%%%%%%%%%%%%%%%%%%%%
\section{Basic definitions}\label{sec:defn}

\subsection{Lens spaces and Dehn surgeries}
On a torus $T$, two essential simple closed curves have {\em distance} $d$ if they may be isotoped to minimally intersect $d$ times.  Given a $3$--manifold $M$ with a torus boundary component $T$, the attachment of a solid torus $V$ to $M$ by an identification of $\bdry V$ and $T$ to produce another $3$--manifold is a {\em Dehn filling} of $M$ along $T$.     If $K$ is a knot in $M$, then a {\em Dehn surgery} on $K$ is a Dehn filling on $M-N(K)$ along $\bdry N(K)$.  The distance of a Dehn surgery or between two Dehn fillings is the distance on the relevant boundary torus of the meridians of the solid tori being attached.

A {\em lens space} is a $3$--manifold that may be expressed as the union of two solid tori joined along their torus boundaries.  Both $S^3$ and $S^1 \times S^2$ are lens spaces.  The two solid tori are called the {\em Heegaard solid tori} of the lens space and their common boundary torus is the {\em Heegaard torus} of the lens space.  Up to isotopy, such a decomposition of a lens space is unique \cite{bonahon,hodgsonrubinstein}.
A {\em (lens space) torus knot} is a knot in a lens space that is isotopic to a curve on the Heegaard torus.  Observe that the core curves of the Heegaard solid tori are special sorts of torus knots.

%%%%%%%%%%%%%%%%
\subsection{Tangles}
In general, a {\em tangle} is the pair of a $3$--manifold and a properly embedded $1$--manifold (often with multiple components), 
though we often speak of the tangle as the $1$--manifold in the $3$--manifold.  Here we will only be concerned with the $3$--manifold being either a $3$--ball or $S^2 \times I$, arising as the $3$--ball with the interior of another deleted.  
Such tangles may be considered equivalent up to homeomorphism, isotopy, or isotopy rel--$\bdry$ (an isotopy of the $1$--manifold fixing its boundary).
%Such tangles will be considered {\em weakly equivalent} if they differ up to a homeomorphism (including mirroring) and {\em strongly equivalent} if they differ only up to isotopy of the $1$-manifold keeping its boundary fixed.   

A {\em marked sphere} is the pair $(S, {\bf x})$ of a sphere $S$ with four marked points ${\bf x}$.  In this article, a {\em (two-strand) tangle} $\tau$ will typically refer to the tangle that is the pair $\tau = (B,t)$ of a ball $B$ with boundary $S$ and a pair of strands $t$ properly embedded in $B$ with boundary ${\bf x}$, usually considered up to strong equivalence.  Notice $\bdry \tau  = (\bdry B, \bdry t) = (S, {\bf x})$.  A {\em trivial tangle} is a two-string tangle $\tau=(B,t)$, regarded up to homeomorphism, in which there is a disk $D$, called the {\em meridional disk} of $\tau$, properly embedded in $B$ disjoint from $t$ that separates $\tau$ into two balls each containing a single unknotted strand.

Observe that an arc $\alpha$ embedded in a tangle $\tau$ with only its endpoints on the strands of $\tau$ has a small regular neighborhood $N(\alpha)$ that intersects $\tau$ in a trivial tangle.  If $\tau$ is a trivial tangle itself, then $\alpha$ is a {\em core arc} of $\tau$ if $\tau-N(\alpha)$ is homeomorphic to $\bdry \tau \times I$, the {\em product tangle} in $S^2 \times I$.  In a fixed trivial tangle $\tau$, two core arcs are always isotopic, keeping their boundaries on the strands of $\tau$.

With an identification of their boundaries to the fixed marked sphere $(S, {\bf x})$ and considered up to isotopy rel--$\bdry$, the trivial tangles are called {\em rational tangles}.  The boundary of a meridional disk of a rational tangle, considered up to isotopy in $S-{\bf x}$, is the {\em meridian} of the rational tangle.

\subsection{Strong inversions}
A link $L$ in $3$--manifold $M$ is said to be {\em strongly invertible} if there is an orientation preserving involution $\iota$ on $M$ such that $\iota(L)=L$ and each component of $L$ intersects the fixed set of $\iota$ twice. The involution $\iota$ is then said to be a {\em strong involution} for $L$.  If the fixed set of $\iota$ meets each torus component of $\bdry M$ in four points, then we say $\iota$ is a strong involution for $M$ even if $L = \emptyset$.   This involution induces a tangle $\tau$ comprised of the $3$--manifold quotient of $M$ by $\iota$ containing the $1$--manifold image of the fixed set.  If the link $L$ is strongly invertible under $\iota$, then it descends to a collection of arcs embedded in $\tau$ with neighborhoods intersecting $\tau$ in trivial tangles. 

\subsection{Rational subtangle replacements} \label{sec:rsrdef} 
Two rational tangles $\rho $ and $\rho'$ (with an identification $\bdry \rho = \bdry \rho'=(S,{\bf x})$), with meridional disks $D$ and $D'$ respectively, have {\em distance $d$} if their meridians $\bdry D$ and $\bdry D'$ may be isotoped in $S - {\bf x}$ to minimally intersect $2d$ times. 
Two tangles  $\tau$ and $\tau'$ are related by a {\em distance $d$ rational subtangle replacement (RSR)} if there is a ball $B_0$ intersecting $\tau$ in a rational tangle $\rho =B_0 \cap \tau$ such that replacing $\rho$ with a rational tangle $\rho'$ of distance $d$ produces $\tau'$.  Such a ball $B_0$ may be determined by an arc $\alpha$ meeting the strands of $\tau$ only at its endpoints (by taking $B_0$ to be a small closed regular neighborhood of the arc $\alpha$).  We refer to both $B_0$ and $\alpha$ as the {\em site} of the RSR on $\tau$.
 Moreover, in the branched double cover of $\tau$, the arc $\alpha$ lifts to a knot $K$ with a solid torus neighborhood that is the lift of $B_0$.  If $\alpha$ is the core arc of the initial rational tangle $\rho$ in an RSR $\rho \mapsto \rho'$ of distance $d$ taking $\tau$ to $\tau'$, then the branched double cover of $\tau'$ may be obtained by the distance $d$ Dehn surgery on $K$ corresponding to the branched double cover of  $\tau$.

\subsection{Plats and continued fractions}
Now view the $3$--ball $B$ as the $1$--point compactification of lower half-space $\{ z \leq 0\} \subset \R^3$ and place the four marked points ${\bf x}$ on the $x$--axis at $(i,0,0)$ for $i=1,2,3,4$.  A rational tangle in $B$ with endpoints at ${\bf x}$ may then be arranged so that its $z$--coordinates have only two local minima.  It may then further be arranged into an {\em open $4$--plat} form with projection to the $xz$--plane as shown on the left of Figure~\ref{fig:rationalplat}.  The oblong rectangles labeled with integers indicate twist regions where the longer direction of a rectangle gives the twist axis.  Its integer gives the number half-twists where the sign determines the handedness, as illustrated on the right of Figure~\ref{fig:rationalplat}.  We always assume the first twist region at the top occurs between the first two strands, permitting $0$ twists if needed.  Reading downwards, we obtain the sequence of integers $\{a_1, a_2, \dots, a_k\}$ which are the coefficients of a continued fraction expansion of a rational number (including $\infty$):
\[ [a_1, a_2, \dots, a_k] = a_1 - \cfrac{1}{a_2 - \cfrac{1}{\ddots - \cfrac{1}{a_k}}} \in \Q \cup \{\infty\}. \]

Observe the usage of minus signs in the continued fraction.  Furthermore notice we may take $k$ to be odd or even as needed since $[\dots,a,\pm1] = [\dots,a\mp1]$.     (Figure~\ref{fig:darcysumnersfig} compares our conventions with the ones used in \cite{darcysumners} and their Figure~1).  By Conway  \cite{conway}, two open $4$--plats describe the same rational tangle if and only if their associated continued fraction expansions represent the same rational number, including $\infty=1/0$. Thus we may use a rational number as well as (the sequences of coefficients of) its continued fraction expansions to describe a rational tangle.

%\begin{figure}
%\centering
%\includegraphics{rationalplat.eps}
%\caption{}
%\label{fig:rationalplat}
%\end{figure}
%
%\begin{figure}
%\centering
%\includegraphics{darcysumnersfig.eps}
%\caption{}
%\label{fig:darcysumnersfig}
%\end{figure}
%

\begin{lemma}\label{lem:palindromecontinuedfraction}
If $[c_1,c_2, \dots,c_n] = a/b$ then
$[c_1, c_2, \dots, c_n, d, -c_n, \dots, -c_2, -c_1] = \frac{da^2}{1+dab}$. 
\end{lemma}
\begin{proof}
See e.g.\ \cite[Corollary 8]{kohn} and adjust for the difference in continued fraction notation.
\end{proof}

\subsection{Pairs of curves on a torus, pairs of rational tangles}
A rational tangle $\tau = (B,t)$ is determined by its meridional disk, and hence by the isotopy class of its meridian on $S-{\bf x}$ where $\bdry \tau = (S, {\bf x})$.  Then through the double cover of $S$ branched over ${\bf x}$, a rational tangle is equivalent to the isotopy class of an unoriented essential curve on the torus and to a rational number that is the slope of this curve after fixing a basis for the torus.  

Recall that the modular group $\Gamma$ of linear fractional transformations  of the upper complex plane $\{z \in \C\, |\, {\it Im} (z) \geq 0 \}$ is isomorphic to $PSL_2(\Z) = \left\{ \left(\begin{array}{cc}a&b\\c&d\end{array}\right) \Big\vert\, ad-bc = 1 \right\} / \pm I$.  That is, for $\phi \in \Gamma$, $\phi(z) = \frac{az+b}{cz+d}$ for integers $a,b,c,d \in \Z$ such that $ad-bc = 1$.

\begin{lemma}\label{lem:pairsofcurves}
Let $\gamma$ and $\gamma'$ be essential curves on the torus $T$ of slope $1/0$ and $r/s$ respectively.  
Then the set of slopes of pairs of essential curves on $T$ homeomorphic to the unordered pair $\{\gamma, \gamma'\}$ is 
\[
 \left\{\Big(\phi(\frac{1}{0}),\phi(w)\Big)   \,\Big\vert\,   w\in\Big\{\pm\frac{r}{s},\pm\frac{s'}{s}\Big\}, \phi \in \Gamma \right\}
\]
where $s' \equiv r^{-1} \mod s$.
\end{lemma}

\begin{proof}
Not worrying about order, we consider both pairs of slopes $(1/0,r/s)$ and $(r/s,1/0)$.  Mirroring takes these to $(1/0,-r/s)$ and $(-r/s, 1/0)$, noting that $-1/0 = 1/0$.   The remaining pairs of slopes of curves homeomorphic to the slopes of $\gamma$ and $\gamma'$ are related to one of these four by orientation preserving homeomorphisms.

The elements of $\Gamma$ give the actions of orientation preserving homeomorphisms of the torus $T$ upon slopes of essential curves.  (This may be seen by the fact that the orientation preserving changes of basis for $H_1(T)$ are given by $SL_2(\Z)$ and the slope of a curve on $T$ is the same for both orientations of the curve.)  The result of the lemma now follows, observing that the map $\phi \in \Gamma$ where $\phi \colon z \mapsto \frac{s' z - r'}{-s z + r}$ takes the pair $(r/s, 1/0)$ to $(1/0,-s'/s)$  where $r',s'$ are a pair of integers such that $rs'-r's=1$, and hence $r s' \equiv 1 \mod s$.
\end{proof}

Observe now that by a mirroring that leaves $1/0$ slope invariant, the pair of the $1/0$--tangle and the $r/s$--tangle is homeomorphic to the pair of the $1/0$--tangle and the $-r/s$--tangle.   A change of basis sending the $r/s$ slope to the $1/0$ slope provides a homeomorphism of the first pair to the pair of the $-s'/s$--tangle and the $1/0$--tangle where $s' \equiv r^{-1} \mod s$.  Similarly we also find a homeomorphism to the pair of the $s'/s$--tangle and the $1/0$--tangle.  In general we have:

\begin{lemma}\label{lem:pairsoftangles}
The pair of the $p/q$--tangle and the $u/v$--tangle is homeomorphic to the pair of the $1/0$--tangle and the $r/s$--tangle for every 
\[
 u/v \in \left\{  \frac{pr + \epsilon p's}{qr + \epsilon q's},  \frac{ps' + \epsilon p's}{qs' + \epsilon q's} \,\Big\vert\, p',q' \in \Z, pq'-qp'=1, \epsilon = \pm1 \right\}
\]
where  $s' \equiv r^{-1} \mod s$.
\end{lemma}

\begin{proof}
This results from Lemma~\ref{lem:pairsofcurves} by taking double branched covers.  Note that every $\phi \in \Gamma$ such that $\phi(1/0)=p/q$ has the form $\phi(z) = \frac{pz+p'}{qz+q'}$ for some $p',q' \in \Z$ with $pq'-p'q=1$. (Fixing one choice of $p',q'$, the others are $p'+Np, q'+Nq$ for $N \in \Z$.) Then $\phi(w)$ takes on the stated forms.
\end{proof}

\section{RSR between rational tangles} \label{sec:RSR}

\begin{thm}\label{thm:mainthm}
Assume a distance $d>0$ RSR from $\rho$ to $\rho'$ takes the rational tangle $\tau$ to the rational tangle $\tau'$.  Then the pair $\{\tau,\tau'\}$ is homeomorphic to the pair $\{1/0,r/s\}$ where $r/s$ belongs to one of the following families:

\begin{itemize}
\item[\0.]  $a/d$ for $a,d$ coprime to and the RSR is the full replacement,
\item[\I.]  $\frac{1+dab}{da^2}$ for $a,b$ coprime,
\item[\II.] $(d=1)$ $\frac{1+4ab}{4a^2}$ for $a,b$ coprime,
\item[\III.] $(d=1)$   
$  \frac{(b-1)(4ab-4a-2b-1)}{(2ab-2a-b)^2}$ for $a,b \in \Z$, or
\item[\IV.] $(d=1)$
$ \frac{(2a-1)(2ab+a-b+1)}{(2ab+a-b)^2}$ for $a,b \in \Z$.
\end{itemize}
\end{thm}

\begin{proof}
The proof breaks into two cases according to the distance $d$.  
Proposition~\ref{prop:d>1} gives the result for $d>1$.  Proposition~\ref{prop:d=1} gives the result for $d=1$.
\end{proof}

\subsection{Distance $2$ and greater RSR} \label{sec:d>=2}
Here we yoke the Montesinos trick \cite{montesinos} to a result of Gabai \cite{gabai} about distance $d \geq 2$ surgeries on knots in solid tori producing solid tori to obtain the classification of distance $d \geq 2$ RSR between rational tangles.  Note that a trivial knot (one that bounds a disk) in a solid torus is isotopic to a meridian of that solid torus and thus may be regarded as a torus knot.  We refer to it as the {\em meridional torus knot} and reserve the term {\em trivial torus knot} for those isotopic to the core of the solid torus, the $(1,n)$--torus knots.

 \begin{thm}[Proof of Lemma~2.3 \cite{gabai}]\label{thm:gabai}
If non-integral Dehn surgery on a knot in a solid torus yields a solid torus, then the knot is a torus knot. %\qed
\end{thm}

The above theorem with Example 3.1 of \cite{gabai2} (attributed to Seifert) gives the following:

\begin{cor}\label{cor:gabai}
Assume distance $d\geq2$ surgery on a knot $K$ in the solid torus $V$ yields a solid torus $V'$. 
If $K$ is not isotopic to be the core of $V$ then $K$ may be isotoped to be a torus knot on a concentric torus in $V$ so that with respect to the induced framing on $K$ induced by this torus, the surgery slope is $\pm 1/d$. %\qed
\end{cor}   

\begin{prop}\label{prop:d>1}
Assume a distance $d\geq2$ RSR from $\rho$ to $\rho'$ takes the rational tangle $\tau$ to the rational tangle $\tau'$.  Then the pair $\{\tau,\tau'\}$ is homeomorphic to the pair $\{1/0,r/s\}$ where $r/s$ belongs to one of the following families:
\begin{itemize}
\item[\0.]  $\frac{a}{d}$ for $a$ coprime to $d$ and the RSR is the full replacement, or
\item[\I.]  $\frac{1+dab}{da^2}$ for $a,b$ coprime
\end{itemize}
\end{prop}

\begin{proof}
(See our paper \cite{BBconf} for a variation of this proof.)
Assume we have a distance $d\geq 2$ RSR taking $\rho \subset \tau$ to $\rho' \subset \tau'$ where $\tau$ and $\tau'$ are rational tangles.      
Taking a branched double cover, this corresponds to a distance $d$ surgery on a knot $K$ in a solid torus $V$ producing a solid torus $V'$.  By Theorem~\ref{thm:gabai}, $K$ is isotopic to a torus knot in $V$, an essential curve on a concentric torus $T_0$ in $V$.   By Corollary~\ref{cor:gabai}, either $K$ is further isotopic to the core of $V$ or, with respect to the framing $K$ inherits from $T_0$, the surgery slope is $\pm 1/d$.   At the expense of swapping the roles of $\tau$ and $\tau'$ (a homeomorphism of the pair), we may assume this surgery is $1/d$.

Under the covering map, $T_0$ descends to a concentric sphere $S_0$ in $\tau$, dividing $\tau$ into an ``inner'' trivial tangle and an ``outer''  product tangle.  With $K$ positioned on $T_0$ to be invariant under the strong involution, the corresponding core arc of $\rho$ (the site to which $K$ descends) lies on $S_0$ and the meridional disk $D$ of $\rho$ meets $S_0$ in a single arc.  Since a component of the lift of the curve $S_0 \cap \bdry \rho$ in the branched double cover gives the framing that $K$ inherits from $T_0$, the $1/d$--surgery on $K$ translates to performing $d$ half-rotations on the disk $S_0 \cap \rho$, twisting the strands of $\tau$.  This insertion of $d$ twists is the replacement of $\rho $ by $\rho'$. 

The tangle $\tau$ admits an open $4$--plat presentation in which $S_0$ is a horizontal level (constant $z$--coordinate), the core arc of $\rho$ is a line segment on $S_0$ between the first two strands, and the replacement of $\rho$ by $\rho'$ is the insertion of $d$ twists.    By a homeomorphism $h$ we may take 
 \[h(\tau) = \frac{1}{0}= [0,0] = [0, c_1, c_2, \dots, c_n, 0, -c_n, \dots, -c_2,-c_1]\]
for any sequence of integers $c_1, c_2, \dots, c_n$.   We may choose $n$ odd  and these $c_i$ so that the middle twist region labeled $0$ between $c_n$ and $-c_n$ indicates the location of $\rho$ and its core arc.  (Taking $n$ odd simply puts this twist region on the first two strands.)  Let $[c_1, c_2, \dots, c_n]=a/b$.  Then the distance $d$ RSR produces 
\[ h(\tau') = [0,c_1, c_2, \dots, c_n,  d, -c_n, \dots, -c_2,-c_1]=-\frac{1+dab}{da^2}, \]
using Lemma~\ref{lem:palindromecontinuedfraction} for that last equality.
By a further orientation reversing homeomorphism, we see that $\{\tau, \tau'\}$ is homeomorphic to $\{\frac{1}{0},\frac{1+dab}{da^2}\}$, giving family \I.

If $K$ is further isotopic to the core of $V$ then $V - N(K) \cong T^2 \times I$ and so $\tau-\rho$ is (weakly) isotopic to the product tangle.  This isotopy guides an isotopy taking the core arc of $\rho $ to the core arc of $\tau$.   Hence the RSR on $\rho$ is effectively just the entire exchange of $\tau$ for $\tau'$, giving family \0.
\end{proof}

\subsection{Distance $1$ RSR}

\begin{prop}\label{prop:d=1}
Assume a distance $d=1$ RSR from $\rho$ to $\rho'$ takes the rational tangle $\tau$ to the rational tangle $\tau'$.  Then the pair $\{\tau,\tau'\}$ is homeomorphic to the pair $\{1/0,r/s\}$ where $r/s$ belongs to one of the following families:

\begin{itemize}
\item[\0.]  $\frac{a}{1}$ for $a\in\Z$ and the RSR is the full replacement,
\item[\I.]  $\frac{1+ab}{a^2}$ for $a,b$ coprime,
\item[\II.]  $\frac{1+4ab}{4a^2}$ for $a,b$ coprime,
\item[\III.]   
$ \frac{(b-1)(4ab-4a-2b-1)}{(2ab-2a-b)^2}$ for $a,b \in \Z$, or
\item[\IV.]  
$ \frac{(2a-1)(2ab+a-b+1)}{(2ab+a-b)^2}$ for $a,b \in \Z$.
\end{itemize}
\end{prop}

%\begin{figure}
%\centering
%\includegraphics[angle=90]{distanceonetangles.eps}
%\caption{}
%\label{fig:d=1}
%\end{figure}

\begin{remark}
This proposition follows from  Berge's classification of  the knots $K$ in a solid torus $V$ with their distance $1$ surgery slopes that transform $V$ into another solid torus $V'$ and partitions them into six families \I\ - \VI \cite{berge}. His Lemma~2.3 describes the slope of the meridian of $V'$ in terms of a standard basis for $\bdry V$, given certain ``homology coordinates'' for these families of surgeries.  Our present proposition and its proof may be viewed as assembling (and perhaps clarifying) these results while also removing some redundancies, as well as translating these into the language of tangles.  \end{remark}

\begin{proof}
If a rational tangle is obtained from a distance $1$ RSR on a rational tangle, then it corresponds by the Montesinos Trick \cite{montesinos} to a distance $1$ surgery on a knot in a solid torus producing a solid torus.  As mentioned in the remark above,
Berge classifies and partitions into six families \I\ - \VI\ the knots in a solid torus $V$ with their distance $1$ surgery slopes that transform $V$ into another solid torus $V'$ \cite{berge}.  These knots are all strongly invertible since each may be viewed as a non-separating curve on the Heegaard surface of a genus $2$ decomposition of the solid torus.   Quotienting by the strong involution and keeping track of the surgery slopes, one may recover a classification of the distance $1$ RSR between rational tangles.  

In \cite{berge-lens} Berge gave a conjecturally complete \cite[Problem 1.78]{kirby}
list of knots in $S^3$ with distance $1$ surgeries producing lens spaces and partitioned them into a dozen families.   The first six families are obtained from his earlier classification in \cite{berge} (mentioned above)  by the various attachments of a second solid torus $W$ to $V$ producing $S^3$.  (The latter six families of these twelve do not arise in this manner.  See Section~\ref{sec:4platadj} for an example.)  

Baker \cite{baker1,baker2} obtains the corresponding tangle descriptions of these twelve families and their surgeries.  
In the first six of these tangle descriptions (in \cite{baker2}) the rational tangles $\omega$ and $\tau$ with subtangle $\rho \subset \tau$ corresponding to the solid tori $W$ and $V$ and the knot $K \subset V$ respectively are easily identifiable. 
Thus we may  determine a distance $1$ RSR exchanging $\rho$ for $\rho'$ that takes the rational tangle $\tau$ to the rational tangle $\tau'$ as follows.

Figures~\ref{fig:bergeI} through \ref{fig:bergeVI} show the progression from the tangle descriptions of the first six families in \cite{baker2} without the rational tangle $\omega$ to a presentation more amenable to observing plat presentations of the distance $1$ RSR.  As shown, these are tangles $\tau - \rho$ in $S^2 \times I$ meeting each boundary component $4$ times whose branched double cover is $V-N(K)$.  The thick black boundary corresponds to the boundary of the solid torus $V$ containing the knot $K$.  We permit the ends of the tangle on this boundary component to move freely as we will account for such changes later.   The small thin boundary component corresponds to $\bdry N(K)$ and has a distance $1$ pair of fillings $\rho$ and $\rho'$, shown above or at its side, that each yields a rational tangle.  To preserve the property of the fillings, any movement of the ends on the thin boundary component is compensated in the filling tangles.  In Figures~\ref{fig:bergeI} and \ref{fig:bergeII}, the large thin circle contains a rational tangle and the rectangle contains an associated $3$--braid with auxiliary fourth strand.    In Figures~\ref{fig:bergeIII}, \ref{fig:bergeIV}, \ref{fig:bergeV}, \ref{fig:bergeVI} the oblong rectangles again denote twist regions containing the specified number of twists ($a,b,a',b',A,B,B' \in \Z$) where the longer direction determines the twist axis and the sign determines the twist handedness as indicated on the right-hand side of Figure~\ref{fig:rationalplat}.  

When performing a Dehn surgery on one knot in a manifold, the core of the attached solid torus is another knot in the resulting manifold.  The new knot is called the {\em dual knot} and admits a surgery (of the same distance) returning the original manifold.  As we are concerned with knots in solid tori with surgeries yielding solid tori, note that the dual knot is again a knot in a solid torus with a surgery yielding a solid torus.  Observe that the dual knots for solid torus surgeries on torus knots are again torus knots, giving family \I (which extends family \I\ of Proposition~\ref{prop:d>1}).  Berge shows that (up to mirroring) the dual knots for family \II\ belong to family \II, \IV\ to \IV, \V\ and \VI\ to \III, and \III\ to either \III, \V, or \VI, see Table~2 of \cite{berge}.  Considering the pair of a knot and its dual, we may include families \V and \VI\ into family \III.   We explicitly demonstrate this in Figures~\ref{fig:bergeV} and \ref{fig:bergeVI} on the tangle level by showing the tangles in $S^2 \times I$ of families \V\ and \VI\ belong to family \III.   Therefore we actually only need to consider the four families \I, \II, \III, and \IV.

%
%\begin{figure}
%\centering
%\includegraphics{BergeI2.eps}
%\caption{}
%\label{fig:bergeI}
%\end{figure}
%
%
%\begin{figure}
%\centering
%\includegraphics{BergeII2.eps}
%\caption{}
%\label{fig:bergeII}
%\end{figure}
%
%\begin{figure}
%\centering
%\includegraphics{BergeIII.eps}
%\caption{}
%\label{fig:bergeIII}
%\end{figure}
%
%\begin{figure}
%\centering
%\includegraphics{BergeIV.eps}
%\caption{}
%\label{fig:bergeIV}
%\end{figure}
%
%
%\begin{figure}
%\centering
%\includegraphics{BergeV.eps}
%\caption{}
%\label{fig:bergeV}
%\end{figure}
%
%\begin{figure}
%\centering
%\includegraphics{BergeVI.eps}
%\caption{}
%\label{fig:bergeVI}
%\end{figure}

In Figures~\ref{fig:bergechangeIandII}, \ref{fig:bergechangeIII}, and \ref{fig:bergechangeIV} we now fix the endpoints of the tangle at the thick boundary, insert the two fillings to produce a pair of rational tangles for that family, and rearrange the resulting rational tangles into a 4-plat presentation so that we may read off their corresponding continued fraction descriptions and hence their associated rational number. 
This then produces a representative pair of RSR distance $1$ tangles in each family.    At the end of the last two of these figures we apply a sequence of flype moves on rational tangles to transfer twistings on the rightmost two strands to the leftmost two strands as shown in Figure~\ref{fig:flypeequivalence}.  One flype transfers the twist to the leftmost two strands at the expense of flipping over the big tangle $R$, but since $R$ is a rational tangle a further sequence of flypes returns $R$ to its original presentation. 

We may then obtain the following continued fractions for our representative pairs:
\[
\begin{array}{lrcl}
\I. &[0, -c_n, -c_{n-1}, \dots, -c_2, -c_1] &\longleftrightarrow &[1, -c_n, -c_{n-1}, \dots, -c_2, -c_1] \\
\II. &[2, -c_n, -c_{n-1}, \dots, -c_2, -c_1] &\longleftrightarrow &[-2, -c_n, -c_{n-1}, \dots, -c_2, -c_1]\\
\III. &[1, b, 2, a]&\longleftrightarrow&[-1, a, 1, b]\\
\IV. &[a, -1, 1, b+1]& \longleftrightarrow&[1, b, -2, a-1]
\end{array}
\]
for integers $a,b,c_1, \dots, c_n$ with $n$ odd (where the wide rectangles in Figure~\ref{fig:bergechangeIandII} contain the sequence of twists $-c_n, -c_{n-1}, \dots, -c_2, -c_1$ beginning with $-c_n$ twists on the middle two strands and proceeding downwards).

Inserting the same sequence of twists at the boundary of each pair of rational tangles produces a homeomorphic pair in the same family.  We do so to the pairs in the families listed above taking the left hand side to the $1/0$--tangle with empty continued fraction  $[\,]=[0,0]$.
\[
\begin{array}{lrcll}
\I. &  1/0=[\,]  &\longleftrightarrow & [0, c_1, \dots, c_n,  1, -c_n, -c_{n-1}, \dots, -c_2, -c_1]&=-\frac{1+ab}{a^2}\\
\II. & 1/0=[\,]  &\longleftrightarrow & [0, c_1, \dots, c_n,  4,-c_n, -c_{n-1}, \dots, -c_2, -c_1]&=-\frac{1+4ab}{4a^2}\\
\III. & 1/0=[\,]  &\longleftrightarrow & [0,-a,-2,-b,-2, a, 1, b]&=
\frac{(b-1)(4ab-4a-2b-1)}{(2ab-2a-b)^2}\\
\IV. & 1/0=[\,]  &\longleftrightarrow & [0,-b-1, -1, 1, -a +1, b, -2, a-1]&=
\frac{(2a-1)(2ab+a-b+1)}{(2ab+a-b)^2}
\end{array}
\]
For \I\ and \II\ we take $a/b = [c_1, \dots, c_n]$ and apply Lemma~\ref{lem:palindromecontinuedfraction}.
By mirroring the pair, we remove the minus on the right hand side in \I\ and \II\ above.  The statement of the proposition now follows.
\end{proof}

\begin{remark}\label{rem:d=1}
A distance $1$ RSR may be viewed as a banding along the core arc of the site of the initial rational tangle.  A framing of the core arc then indicates the particular distance $1$ RSR.
Up to mirroring, Figure~\ref{fig:d=1} illustrates Proposition~\ref{prop:d=1} with $\tau=1/0$ where the thick grey arc gives the site of the RSR and is framed by the plane of the page.  Here, $\gamma$ indicates the sequence of twists $c_1, c_2, \dots, c_n$ (where $[c_1, c_2, \dots, c_n]=b/a$) running downwards starting with the middle two strands while $\bar{\gamma}$ indicates the inverse sequence $-c_n, \dots, -c_2, -c_1$.
Note the arc at the beginning and end of set \I\ may be isotoped to be horizontal, though it would no longer be framed by the plane of the page.    
\end{remark}

%
%\begin{figure}
%\centering
%\includegraphics{flypeequivalence2.eps}
%\caption{}
%\label{fig:flypeequivalence}
%\end{figure}
%
%\begin{figure}
%\centering
%\includegraphics{BergechangeIandIIflype2.eps}
%\caption{}
%\label{fig:bergechangeIandII}
%\end{figure}
%
%\begin{figure}
%\centering
%\includegraphics{BergechangeIIIflype2.eps}
%\caption{}
%\label{fig:bergechangeIII}
%\end{figure}
%
%\begin{figure}
%\centering
%\includegraphics{BergechangeIVflype2.eps}
%\caption{}
%\label{fig:bergechangeIV}
%\end{figure}

\subsubsection{Multiple distance $1$ RSR}\label{sec:bergetangle}
As discussed above:
Every surgery on the core of a solid torus $V$ produces another solid torus.  Every torus knot in the solid torus $V$ has an integral family of surgeries producing a solid torus, those surgeries distance $1$ from the framing induced by the concentric torus in $V$ in which they sit.  Furthermore, if a torus knot is not actually the core of $V$, then these are all the surgeries to another solid torus.  Indeed no triple of these surgeries on torus knots are mutually distance $1$ from one another. Correspondingly, the sites for RSR in family \I\ admit multiple rational tangle fillings, but no triple of them, except when the site is the core of the tangle (family \0), are mutually RSR distance $1$.

Gabai shows that any knot in a solid torus that is not a torus knot admits at most three solid torus surgeries and that these surgeries are mutually distance $1$  \cite{gabai}. 
Berge shows that there is a non-torus knot $K$ in a solid torus $V$ that has two distance $1$ surgeries yielding solid tori that are also distance $1$ from one another, and, moreover, that this knot is unique up to homeomorphism of $V$ \cite{berge}.  To rephrase, other than $T^2 \times I$, the manifold $M=V-N(K)$ is the only $3$--manifold such that along a particular torus of $\bdry M$ there are three Dehn fillings of slopes all distance $1$ from one another that each result in a solid torus. The manifold $M$ has come to be known as the {\em Berge manifold}, e.g.\ \cite{martellipetronio}, \cite{hoffman}.  

We present a tangle in $S^2 \times I$, the {\em Berge tangle}, such that along a component of $\bdry (S^2 \times I)$ there are three rational tangle fillings all distance $1$ from one another, each yielding a rational tangle.  In Figure~\ref{fig:BergeTangle}, the leftmost image is the Berge tangle and the three rows to its right show three fillings and their isotopies to open $4$--plat presentations.  Figure~\ref{fig:tangleS1xS2} shows another presentation of this tangle from which its $3$--fold symmetry is more apparent; the three distance $1$ fillings are related by this symmetry.

If this Berge tangle were actually homeomorphic to the product tangle then every filling of its two boundary components with rational tangles would produce a $2$--bridge knot or link, yet this is not the case since its filling shown in Figure~\ref{fig:8-17} is the knot $8_{17}$ in Rolfsen's table \cite{rolfsen} which is not $2$--bridge.  It then follows that the branched double cover of the Berge tangle is the Berge manifold.  

Our Theorem~\ref{thm:hyperbolicRSR} shows this Berge tangle is the only tangle in $S^2 \times I$ with the Berge manifold as its branched double cover, since the relevant knot $K$ in $V$ is hyperbolic.   
\begin{cor}
The Berge tangle is the only tangle in $S^1 \times S^2$ such that along a component of $\bdry (S^2 \times I)$ there are three rational tangle fillings all distance $1$ from one another, each yielding a rational tangle. \qed
\end{cor}

%\begin{figure}
%\centering
%\includegraphics{BergeTangle.eps}
%\caption{}
%\label{fig:BergeTangle}
%\end{figure}
%
%
%\begin{figure}
%\centering
%\includegraphics[height = 2in]{BergeS1xS2.eps}
%\caption{}
%\label{fig:tangleS1xS2}
%\end{figure}
%
%\begin{figure}
%\centering
%\includegraphics{8-17.eps}
%\caption{}
%\label{fig:8-17}
%\end{figure}

%%%%%%%%%%%%%%%%
\section{Distance $1$ RSR between $2$--bridge links}\label{sec:2bridge}
Darcy-Sumners \cite{darcysumners} determines which $2$--bridge links are related by an RSR of distance at least $2$; these results are summarized within Theorem~\ref{thm:sites2bridge}.  The question of which $2$--bridge links are related by an RSR of distance $1$ is still open as the related question of which pairs of lens spaces are related by a distance $1$ Dehn surgery on a knot is still open, though many examples may be obtained by generalizing Berge's works \cite{berge,berge-lens}.  However, the recent works of Greene \cite{greene} and Lisca \cite{lisca} (that is, as Green remarks Rasmussen observed \cite{greene})  classify which lens spaces may be obtained from $S^3$ or $S^1 \times S^2$, respectively,  by distance $1$ surgery on a knot. (Note, a family is missing from the statement \cite[Lemma~7.2]{lisca} though it arises in the proof; consequently the corresponding lens spaces were not included in the main result.)  As these surgeries are all realized on strongly invertible knots, their works determine which $2$--bridge links are related by a distance $1$ RSR to the unknot or the $2$--component unlink.  We use the notation $S(p,q)$ to denote the normal closure of the $p/q$--tangle and use the continued fractions for $p/q$ to describe $S(p,q)$ too.

\begin{thm}[Greene \cite{greene}]
The $2$--bridge link $K$ is related to the unknot by a distance $1$ RSR if and only if there exist $p>q>0$ such that $K \cong S(p,q)$  and for some $k \in \Z$, both $q \equiv \pm k^2 \mod p$ and one of the following holds:
\begin{enumerate}
\item $p \equiv ik\pm q \mod k^2 \quad \quad \gcd(i,k)=1 \mbox{ or } 2$
\item $\begin{cases} p \equiv \pm(2k-1)\delta \mod k^2, &\delta| k+1, \frac{k+1}{\delta} \mbox{ odd} \\
p \equiv \pm(2k+1)\delta \mod k^2, &\delta| k-1, \frac{k-1}{\delta} \mbox{ odd} 
\end{cases}$
\item  $\begin{cases} p \equiv \pm(k-1)d \mod k^2, &\delta| 2k+1 \\
p \equiv \pm(k+1)\delta \mod k^2, &\delta| 2k-1  
\end{cases}$
\item  $\begin{cases} p \equiv \pm(k+1)\delta \mod k^2, &\delta| k+1,\delta\mbox{ odd} \\
p \equiv \pm(k-1)\delta \mod k^2, &\delta| k-1,\delta\mbox{ odd}
\end{cases}$
\item $\pm k^2 + k + 1 \equiv 0 \mod p$
\item $p=\frac{1}{11}(2k^2+k+1), k\equiv 2  \mbox{ or } 3 \mod 11$.
\end{enumerate}
\end{thm}

\begin{proof}
In \cite{greene}, Greene shows that every lens space that may be obtained by integral surgery on a knot in $S^3$ is also obtained by the same surgery on a Berge knot.  Since all Berge knots are strongly invertible, his result implies this theorem.  The conditions on the $p,q$ shown are adapted from Rasmussen's list in \cite{rasmussen} (and in \cite{greene}) describing the lens spaces that may be obtained from the various families of Berge knots. 
\end{proof}

\begin{thm}[Lisca \cite{lisca}, Rasmussen \cite{greene}]
The $2$--bridge link $K$ is related to the $2$--component unlink by a distance $1$ RSR if and only if either $K$ is the unknot, $K$ is the unlink, or there exist $p>q>0$ such that $K \cong S(p,q)$, $p=m^2$ for some $m \in \N$, and one of the following holds:
\begin{enumerate}
\item $q=mk\pm1$ with $m > k > 0$ and $\gcd(m,k)=1$;
\item $q=mk\pm1$ with $m > k > 0$ and $\gcd(m,k)=2$;\footnote{This family is missing in the statements of \cite{lisca} though it arises in a proof. It corresponds to family II of the Berge knots, certain cables on torus knots.}  
\item $q=\delta(m\pm1)$ where $\delta >1$ divides $2m \mp 1$; or
\item $q=\delta(m\pm1)$ where $\delta>1$ is odd and divides $m\pm1$.
\end{enumerate}
\end{thm}

\begin{proof}
In \cite{lisca}, Lisca classifies which lens spaces bound rational homology balls and correspondingly which $2$--bridge links bound ribbon surfaces.  While for many of these $2$--bridge links he exhibits ribbon surfaces obtained from a single banding (a distance $1$ RSR) to the unlink of two components, he uses two bandings for one family.  In Section~1.5 of \cite{greene}, Greene notes Rasmussen had observed that Lisca's lens spaces arise from considering the Berge knots in solid tori with a solid torus surgery as residing in a Heegaard solid torus of $S^1 \times S^2$.  Since these knots in $S^1 \times S^2$ are all strongly invertible, all these $2$--bridge links admit a distance $1$ RSR to the unlink.  
\end{proof}

\begin{remark}
These RSR from slice $2$--bridge links to the unlink may be observed directly using the appropriate closures of our tangle presentations of Berge's knots in solid tori.  Also, through studying Lisca's embeddings of the intersection lattices associated to lens spaces bounding rational homology balls, Lecuona had also further determined that these $2$--bridge links may be shown to bound ribbon surfaces with a single banding.  In a forthcoming article we are together demonstrating these as well as a two other families of bandings from $2$--bridge links to the unlink, thereby refuting and then amending Conjecture~1.8 of \cite{greene}.
\end{remark}

\subsection{Distance $1$ RSR of $2$--bridge links to the unknot that do not arise from closures of distance $1$ RSR of rational tangles.}\label{sec:4platadj}

Distance $1$ RSR between $2$--bridge links do not all arise from the normal closures of distance $1$ RSR between rational tangles.  Of course normal closures of a pair of rational tangles having a distance $1$ RSR gives a pair of $2$--bridge links having a distance $1$ RSR.

Berge's families of knots {\sc V\!I\!I -- X\!I\!I} \cite{berge-lens}  all have distance $1$ Dehn surgeries transforming $S^3$ into a lens space, but generically do not preserve the genus $1$ Heegaard splittings.  In other words, they generally do not correspond to knots in solid tori with Dehn surgeries producing solid tori.  As these knots are all strongly invertible, they induce distance $1$ RSR between the unknot and $2$--bridge links that cannot be obtained as the closures of distance $1$ RSR of rational tangles.    

\medskip
\noindent{\textbf{Example:}}  Figure~\ref{fig:largevoladjacency} gives an example of a distance $1$ RSR between the $2$--bridge knot $S(1137,430)$ (left) and the unknot (right).  As also shown, the exterior of the knot $K(29,11)$ in Figure~\ref{fig:fatknot} is the branched double cover of the exterior of the RSR site in Figure~\ref{fig:largevoladjacency}.  One may check (e.g.\ with SnapPy \cite{snappy}) that $K(29,11)$ is a hyperbolic knot of volume greater than that of the minimally twisted $5$--chain link (the pretzel link $P(2,-2,2,-2,2)$).  Since hyperbolic volume decreases under surgery \cite{thurston}, its exterior cannot be expressed as fillings of the exterior of the minimally twisted $5$--chain link.    Work of the first author \cite{baker2} shows that any knot with a lens space surgery corresponding to surgery on a knot in solid torus with Dehn surgery producing a solid torus is expressible as a filling of the exterior of the minimally twisted $5$--chain link.  Thus the knot $K(29,11)$  cannot correspond to a knot in a solid torus that admits a Dehn surgery yielding another solid torus.   Consequentially, the distance $1$ RSR of $S(1137,430)$ to the unknot does not arise from the normal closures of a distance $1$ RSR between rational tangles.

In fact the knot $K(29,11)$ may be expressed as a filling of the exterior of the minimally twisted $7$--chain link in a way that corresponds to a sequence of Dehn twists on the fiber of the trefoil.  The notation $K(29,11)$ used here is only to record that the knot represents the homology class $29a+11b$ on the fiber of the trefoil where $a$ and $b$ represent homology classes of the cores of the Hopf bands pulled visibly to the left and the right of Figure~\ref{fig:fatknot}.  The continued fraction expansion of $29/11=[3,3,4]$ records a sequence of  Dehn twists along $a,b,a$ that produces the knot from the curve $b$.  The degree of twisting was chosen to ensure the knot had large enough volume to preclude another description on the minimally twisted $5$--chain.  See \cite{baker1} for details.

%
%\begin{figure}
%\centering
%\includegraphics{largevolumeadjacency.eps}
%\caption{}
%\label{fig:largevoladjacency}
%\end{figure}
%
%\begin{figure}
%\centering
%\includegraphics[width=6in]{fatknot.eps}
%\caption{}
%\label{fig:fatknot}
%\end{figure}
%

%%%%%%%%%%%%%%%%%%
\section{Sites of RSR} \label{sec:sites}

The {\em site} of an RSR transforming $\tau$ to $\tau'$ is an embedded arc $\alpha$ meeting the strands of $\tau$ at its endpoints.   The intersection of $\tau$ with a small regular neighborhood $N(\alpha)$ is a rational tangle $\rho$ which the RSR replaces with the rational tangle $\rho'$ to produce $\tau'$. 

In this section, we determine where these sites must be for RSR both between tangles (Theorem~\ref{thm:sitessimple}) and, for $d \geq 2$, between 2-bridge links (Theorem~\ref{thm:sites2bridge}). 

Figure~\ref{fig:d=1} illustrates, up to homeomorphism, sites where the RSR of families \I\ -- \IV\ may occur and the corresponding transformations in the case $d=1$, see Remark~\ref{rem:d=1}.  This generalizes in the case $d>1$ for family \I.  Self-homeomorphisms of a rational tangle, namely rotations exchanging and/or inverting its strands, may take the site of an RSR in the above theorem to a non-isotopic site.  

\textit{A priori,} there may be other sites inducing these RSR between rational tangles that are not related by a homeomorphism of the initial tangle.   For this to occur there would be sites $\alpha_0$ and $\alpha_1$ in the tangle $\tau$ such that the tangles $\tau-N(\alpha_0)$ and $\tau-N(\alpha_1)$ either (a) have non-homeomorphic branched double covers so that the lifts of the sites correspond to two distinct knots in the branched double cover of $\tau$ with surgeries yielding the same manifold or (b) have (orientation preserving) homeomorphic branched double covers and hence correspond to the same knot in the branched double cover of $\tau$.  The occurrences of situation (a) are covered by the classification of knots in solid tori with surgeries yielding solid tori.  Our present interest lies in situation (b).   This is a concern, as there are distinct links in $S^3$ with homeomorphic branched double covers (for a recent overview, see Mecchia \cite{mecchia}).  However, in our situation Theorem~\ref{thm:sitessimple} proves there are no unexpected sites.

\begin{thm}\label{thm:sitessimple}
The RSR between rational tangles of Theorem~\ref{thm:mainthm} occur, up to homeomorphism, only at the sites of the core arc for family \0\ or as indicated in Figure~\ref{fig:d=1} for families \I\ -- \IV.
\end{thm}

\begin{thm}\label{thm:sites2bridge}
If the $2$--bridge link $S(p,q) = [a_1, a_2, \dots, a_n]$ has a distance $d \geq 2$ RSR to $S(u,v)$ then there exists integers $c_1, c_2, \dots, c_k$ such that 
 \[S(p,q) = [a_1, a_2, \dots, a_n,0, c_1, c_2, \dots, c_k, 0, -c_k, \dots, -c_2, -c_1]\]
 and
\[S(u,v) = [a_1, a_2, \dots, a_n, 0, c_1, c_2, \dots, c_k, \pm d, -c_k, \dots, -c_2, -c_1].\]  
Up to homeomorphism, the site of the RSR is the twist region corresponding to the $0$ between the $c_k$ and the $-c_k$ in the plat associated to the continued fraction for $S(p,q)$ and to $\pm d$ in the plat associated to the continued fraction for $S(u,v)$.
\end{thm}

Theorems~\ref{thm:sitessimple} and \ref{thm:sites2bridge} are proven below and rely on results from Section~\ref{sec:siteproofs}.  
The methods employed depend on the geometry of the branched double cover of the complement of the site. 

When $d \geq 2$ then in the setting of Theorem~\ref{thm:sitessimple} the manifolds are complements of torus knots in solid tori and thus are generalized Seifert fiber spaces.\footnote{ These are all true Seifert fiber spaces except in the case of the exterior of a trivial knot in a solid torus.}  
The same techniques applies to distance $d \geq 2$ RSR between $2$--bridge links.   
Thus we may build upon the classification of pairs of $2$--bridge links with RSR distance $d \geq 2$ given by Darcy-Sumners \cite{darcysumners}, and additionally determine that the distance $d\geq 2$ RSR between $2$--bridge links occur, up to homeomorphism, only at the sites implied by the continued fraction expansions of Theorem~1(3) of \cite{darcysumners} (which may be seen as arising from the normal closure of family~\I).   

When $d=1$ in Theorem~\ref{thm:sitessimple}, these manifolds may be (generalized) Seifert fibered, a toroidal union of two cable spaces, or hyperbolic.  The $d=1$ Seifert fibered cases are the same as the $d\geq 2$ setting.  The case of the toroidal union of cable spaces relies upon an understanding of strong involutions of manifolds with non-trivial JSJ decompositions which we adapt from work of Paoluzzi \cite{paoluzzi}.  For the hyperbolic manifolds we view our tangles as hyperbolic orbifolds and employ results of Dunbar-Meyerhoff on hyperbolic Dehn surgery for orbifolds \cite{dunbar-meyerhoff}  and of Wang-Zhou on symmetries of knots in $S^3$ with lens space surgeries \cite{wang-zhou}.

A theorem analogous to Theorem~\ref{thm:sites2bridge} that treats the $d=1$ setting would require (or produce) a classification of strongly invertible knots in lens spaces with integral surgeries producing lens spaces.

\begin{proof}[Proof of Theorem~\ref{thm:sitessimple}]  
The proofs of Propostions~\ref{prop:d>1} and \ref{prop:d=1} show that in the solid torus branched double cover $V$ of a rational tangle, the site of an RSR of Theorem~\ref{thm:mainthm} lifts to a knot $K$  with a solid torus surgery.
 
Assume $\alpha_0$ and $\alpha_1$ are two such sites lifting to $K$.  The exterior of this knot $M=V-N(K)$ is the branched double cover of each of the tangles $\tau-N(\alpha_0)$ and $\tau-N(\alpha_1)$.  By the classification of Dehn surgeries on knots in solid tori yielding solid tori \cite{berge,gabai,moser}, this exterior $M$ is either  a generalized Seifert fiber space over the annulus with one exceptional fiber, the toroidal exterior of a $(2,1)$--cable of a torus knot in a solid torus, or a hyperbolic manifold.   Theorems~\ref{thm:uniquetangles}, \ref{thm:satelliteRSR}, and \ref{thm:hyperbolicRSR} prove that in each of these three settings there is an orientation preserving homeomorphism of tangles $h \colon \tau - N(\alpha_0) \to \tau-N(\alpha_1)$.   
Then $h$ maps the meridian of $\alpha_0$ to the meridian of $\alpha_1$ if and only if $h$ extends to a homeomorphism of pairs $H \colon (\tau, \alpha_0) \to (\tau, \alpha_1)$.
Thus if this homeomorphism does not extend, then $N(\alpha_0)$ and $N(\alpha_1)$ lift to two different Dehn fillings of $M$ producing $V$.  
In other words there is a non-trivial Dehn surgery on $K$ that yields $V$ (preserving the meridian on $\bdry V$).  By the classification of Dehn surgeries on knots in solid tori yielding solid tori, this only occurs if $K$ is a trivial knot.
\end{proof}

\begin{proof}[Proof of Theorem~\ref{thm:sites2bridge}]
Theorem~1 of Darcy-Sumners \cite{darcysumners} classifies the pairs of $2$--bridge links related by a distance $d\geq2$ RSR.  In the lens space $Y$ that is the branched double cover of the $2$--bridge link $\beta$, any site of a distance $d \geq 2$ RSR to another $2$--bridge link lifts to a non-trivial torus knot $K$.  (By the Cyclic Surgery Theorem \cite{cgls} together with either  Theorem~\ref{thm:seifertsurgery} or the work of \cite{darcysumners}, a distance $d\geq2$ Dehn surgery between lens spaces must occur on a torus knot ).   Assume $\alpha_0$ and $\alpha_1$ are two such sites lifting to (curves isotopic to) $K$.  The exterior of this torus knot $M=Y-N(K)$ is a Seifert fiber space over the disk with two exceptional fibers and is the branched double cover of each of the two-string tangles $\beta - N(\alpha_0)$ and   $\beta-N(\alpha_1)$.

By Theorem~\ref{thm:uniquetangles} there is an orientation preserving homeomorphism of tangles $h \colon \beta - N(\alpha_0) \to \beta-N(\alpha_1)$.  Then $h$ maps the meridian of $\alpha_0$ to the meridian of $\alpha_1$ if and only if $h$ extends to a homeomorphism of pairs $H \colon (\beta, \alpha_0) \to (\beta, \alpha_1)$.
If this homeomorphism does not extend, then $N(\alpha_0)$ and $N(\alpha_1)$ lift to two different Dehn fillings of $M$ producing $Y$.  In particular $K$ admits a non-trivial Dehn surgery that yields $Y$ (with the same orientation).  By the classification of surgeries on torus knots in lens spaces \cite{moser,darcysumners}, this may only happen if $K$ is either a trivial knot or isotopic to the core of a Heegaard solid torus.  (See also Rong \cite{rong}.)  Thus $\alpha_0$ and $\alpha_1$ must lift to the two cores of the Heegaard solid tori, but when these are isotopic the homeomorphism of pairs $H$ exists.
\end{proof}
 
\subsection{Exteriors of sites}\label{sec:siteproofs}

\begin{thm}[Seifert fibered]\label{thm:uniquetangles}
$\quad$
\begin{enumerate}
\item[(I)] If a Seifert fiber space $M$ over the disk with two exceptional fibers is the branched double cover of a tangle $\sigma$  in the $3$--ball, then $\sigma$  may be expressed as the sum of two rational tangles. (Ernst \cite{ernst})
\item[(II)] If a generalized Seifert fiber space $M$ over the annulus with one exceptional fiber is the branched double cover of a tangle $\sigma$ in $S^2 \times I$, then $\sigma$ may be expressed as the sum of a rational tangle and a $4$--braid in $S^2 \times I$. (cf.\ \cite[Lemma~3.8]{gl:unknottingnumber1})
\end{enumerate}

Moreover, in either of Case (I) or (II), if $\sigma_0$ and $\sigma_1$ are two such tangles then they are isotopic (without fixing the boundary).
\end{thm}

Recall that a {\em cable space} is an ordinary (non-generalized) Seifert fiber space over the annulus with a single exceptional fiber.  A cable space with exceptional fiber of order $p>1$ may be identified as the exterior of a $(p,q)$--torus knot in a solid torus for some $q$ coprime to $p$.  Proposition~2.8 of \cite{paoluzzi} implies that a cable space has a unique strong involution; Theorem~\ref{thm:uniquetangles} gives a detailed proof and characterizes its quotient tangle.

\begin{proof}
Case (I) is a special case of Theorem~8 of Ernst \cite{ernst}, and we pattern the general proof of Case (II) after this proof.

First assume $M$ has a Seifert fibration.  
The branched double covering $\pi \colon M \to \sigma$ of the tangle $\sigma$ by $M$ induces an involution on $M$.  
By Tollefson \cite{tollefson} (see also Lemma~2.4.31 of Brin \cite{brin}), such an involution is isotopic to one, say $\iota \colon M \to M$, that preserves the Seifert fibering of $M$.  

Let $A$ be an annulus that is the orbit surface of the Seifert fibering of $M$, and let $\psi \colon M \to A$ be the quotient map that takes each circle fiber to a point on $A$.  Then, since $\iota$ sends fibers to fibers, $\iota$ confers an involution $\iota_A$ upon $A$.  Because the involution has four fixed points on each component of $\bdry M$, and a circle fiber on $\bdry M$ intersects exactly $0$ or $2$ of them,  the involution on $A$ has two fixed points on each component of $\bdry A$.  Let $\bdry_1 A$ and $\bdry_2 A$ be the two components of $\bdry A$, and let $a_i, b_i$ be the two fixed points of $\iota_A$ on $\bdry_i A$ for $i=1,2$.  Observe that $\iota_A$ exchanges the complementary arcs $\bdry_i A - \{a_i,b_i\}$ for each $i=1,2$.  Thus there exist two arcs of fixed points in $A$ each of which must connect the two components of $\bdry A$.  
(If the endpoints of each fixed arc were on their own component of $\bdry A$, then these arcs would chop $A$ into two disks and an annulus. But then $\iota_A$ would have to exchange the two disks with this annulus, a contradiction.)  
By swapping the labels $a_2$ and $b_2$ if necessary,  let  $g_a$ be the arc of fixed points connecting $a_1$ to $a_2$ and let $g_b$ be the one connecting $b_1$ to $b_2$.  Then $g = g_a \cup g_b$ divides $A$ into two disks $D_1$ and $D_2$ that are exchanged by $\iota_A$.   Let $e$ be the point on $A$ that is the image under $\psi$ of the exceptional fiber of $M$.

First assume the point $e$ on $A$ is not on $g$ but rather in, say, the interior of $D_1$.  
Let $\beta$ be a properly embedded arc in $D_1$ separating $g$ from $e$.  
Then cutting $M$ along the two annuli $\psi^{-1}(\beta) \cup \psi^{-1}(\iota_A(\beta))$ separates $M$ into two solid tori and $T^2 \times I$.   One of these solid tori is a fibered solid torus neighborhood of the exceptional fiber, and the other is its image under $\iota$.  Hence the exceptional fiber must have trivial order and the two annuli $\psi^{-1}(\beta)$ and $\psi^{-1}(\iota_A(\beta))$ must be longitudinal on these solid tori.  
 Since the $T^2 \times I$ contains the fixed set of $\iota$ and projects under $\pi$ to the product tangle in $S^2 \times I$ with strands $t$, then the two solid tori project to a single solid torus which meets $S^2 \times I$ in a single annulus $\pi(\psi^{-1}(\beta) \cup \psi^{-1}(\iota_A))$ that is disjoint from the strands $t$ and longitudinal in the solid torus.  Therefore $\sigma$ is the product tangle in $S^2 \times I$.
 
Now assume the point $e$ on $A$ is on $g$.  Let $\beta$ be an arc from $\bdry A$ to $g$ such that $\beta \cup \iota_A (\beta)$ cuts $A$ into a disk $D$ containing $e$ in its interior and an annulus $A'$.  Then $\psi^{-1}(\beta \cup \iota_A (\beta))$ is an annulus meeting the fixed set of $\iota$ in two points (on the circle $\psi^{-1}((\beta\cup \iota_A(\beta)) \cap g)$), $\psi^{-1}(A') = A' \times S^1 \cong T^2 \times I$ meeting the fixed set of $\iota$ in four arcs of the form $\{x\} \times I$, and $\psi^{-1}(D)$ is a solid torus meeting the fixed set of $\iota$ in two arcs that pass through the exceptional fiber.  Thus, up to homeomorphism, under $\pi$ these quotient to a disk with two points, the product tangle in $S^2 \times I$, and a trivial $2$--strand tangle in a ball.  As the disk is in the boundary of each of these tangles, we therefore obtain $\sigma$ as the sum along this disk of a $4$--braid in $S^2 \times I$ with a rational tangle.  

When $M$ has a generalized Seifert fibration, it is the exterior of the trivial knot in the solid torus and hence homeomorphic to the connect sum of two solid tori. By Kim-Tollefson \cite{kimtollefson} the involution on $M$ may be viewed as the connect sum of involutions on the two solid torus summands whose quotients are rational tangles.  In particular, this connect sum of solid tori is done along equivariant regular neighborhoods of points on the fixed sets. (Note that this involution is isotopic to one that preserves the generalized Seifert fibration.)  Consequently the quotient tangle $\sigma$ is homeomorphic to the connect sum of two rational tangles along neighborhoods of points on a strand of each tangle.   This tangle is equivalent to one composed of two $I$--fibers of $S^2 \times I$ and a $\bdry$--parallel arc at each boundary component of $S^2 \times I$. Hence it admits a decomposition as desired.

\smallskip

Now in the decompositions for Cases (I) and (II), each rational tangle  lifts to a  solid torus neighborhood of an exceptional fiber, the tangle of a $4$--braid in $S^2 \times I$ lifts to $T^2 \times I$, and a summing disk lifts to an annulus in the corresponding boundary tori.  The last line of the theorem then follows in Case (I) because the two rational tangles of $\sigma_0$ and $\sigma_1$ must coincide up to twists along the summing disk and their order which all may be changed by a rotation of the ball containing the tangle sum.  These moves may be achieved through isotopy of the strands of the tangle.  (Note if there were $3$ or more exceptional fibers in $M$ there would be multiple non-homeomorphic tangles.)  It follows in Case (II) because the rational tangles of $\sigma_0$ and $\sigma_1$ must coincide up to twists along the summing disk, and after a homeomorphism straightening the $4$--braid to the product tangle, the summed rational tangle may be slid through the product tangle from one boundary sphere to the other.  This too may be achieved through an isotopy. 
\end{proof}

The Jaco-Shalen-Johannson decomposition \cite{jacoshalen,johannson} of an irreducible $3$--manifold $M$ may be viewed as a collection of tori that chops $M$ into hyperbolic and Seifert fibered submanifolds.  This decomposition may be taken to be equivariant with respect to any given strong involution $\iota$ on $M$ \cite{meeksscott, jacorubinstein}. This decomposition of $M$ then quotients under $\iota$ to the Bonahon-Siebenmann decomposition \cite{bonahonsiebenmann} of the orbifold $\tau = (Q,t)$ (which we also view as a tangle) in which the underlying space $Q$ is the quotient $M/\iota$ and the orbifold locus $t$ has order $2$ and is the image of the fixed set of $\iota$.  This decomposition chops $\tau$ along toric suborbifolds: incompressible tori disjoint from $t$ and {\em Conway spheres}, spheres meeting $t$ transversally in four points.  

Paoluzzi's work \cite{paoluzzi} on the number of ``hyperbolic type'' involutions of $3$--manifolds $M$, i.e.\ where the quotient $\tau =(Q,t)$ is the pair of $S^3$ and a hyperbolic knot, may be applied more generally.  This is evidenced by the characterization of hyperbolic type involutions on manifolds with non-trivial JSJ decompositions as those involutions for which the restriction to each piece $M_i$ quotients to give an orbifold topologically equivalent to $S^3$ minus a number of open balls, one for each component of $\bdry M_i$, \cite[Proposition~2.1]{paoluzzi}.   For our purposes, it is enough to observe that this work extends to address strong involutions of manifolds with boundary in which (a) every torus of the equivariant JSJ decomposition is fixed by the involution (and hence descends to a Conway sphere), (b) the base orbifold of the Seifert fibered pieces have genus $0$, and (c) the characteristic graph of the decomposition (a vertex for each piece and an edge for each torus connecting the vertices corresponding to the components that contain it in their boundaries) is a tree.  

Following Paoluzzi, let us say that two strong involutions on $M$ of the kind above are {\em equivalent} if there is a homeomorphism of $M$ that acts as the identity on the characteristic tree of the decomposition and conjugates one involution to the other.  The extension of Paoluzzi's work then yields this next proposition.

\begin{prop}[Paoluzzi, Proposition~2.4 \cite{paoluzzi}\footnote{In the proof of \cite[Proposition~2.4]{paoluzzi}, the author employs \cite[Theorem~2]{tollefsonPLmap} but the citation points to the wrong article.}]\label{prop:paoluzzi}
Let $M=M_1 \cup M_2$ be the union of two cable spaces $M_i$, $i=1,2$, along a torus $M_1 \cap M_2 = T$ so that $M$ is not Seifert fibered itself.  Assume $\iota$ is a strong involution of $M$.  Then there are at most four non-equivalent strong involutions on $M$ whose restrictions to $M_i$ are equivalent to the restrictions of $\iota$. \qed
\end{prop}
  
The four candidate inequivalent strong involutions arise from glueings of $M_1$ to $M_2$ along $T$ that are isotopic to the identity and commute with the restriction of $\iota$ to $T$.  As such they either fix the four points of the intersection of the fixed set of $\iota$ with $T$ yielding an involution equivalent to $\iota$ or freely permute the four fixed points yielding the other three.  In the quotient $\tau$ of $M$ by $\iota$, these correspond to mutation along the Conway sphere $S$ to which $T$ descends  \cite{conway}. 
 As shown in Figure~\ref{fig:mutation}, a {\em mutation} along $S$ may be viewed as doing one of the four alterations (including the identity alteration) to the tangle {\em (left)} in a neighborhood of $S$.   Alternatively a mutation may be regarded as cutting $\tau$ open along $S$ and reattaching by an automorphism of the marked sphere $(S, S \cap t)$ that preserves slopes (unoriented isotopy classes of simple closed curves in $S-t$).  With $S$ as the unit sphere in $\R^3$ and $S \cap t = (\pm\tfrac{1}{\sqrt{2}}, \pm\tfrac{1}{\sqrt{2}},0)$, these automorphisms are the group of the three rotations by $\pi$ through the coordinate axes with the identity. While a mutation does not affect the branched double cover it will produce a different tangle if the automorphism of $(S, S \cap t)$ does not extend to an automorphism of at least one tangle on either side of $S$.

\begin{thm}[Cabled Knots]\label{thm:satelliteRSR}
Let $K$ be a non-trivial connected cable of a (non-trivial, non-meridional) torus knot in a solid torus $V$.  Then there are at most $2$ non-equivalent strong involutions of $V-N(K)$.  If $K$ is a $(2,n)$--cable, then there is only one strong involution.   

Hence there are at most two tangles in $S^2 \times I$  whose branched double covers are orientation preserving  and boundary component preserving homeomorphic to $V-N(K)$ (both with $S^2 \times \{1\}$  lifting to $\bdry V$, say) but themselves are not orientation preserving homeomorphic. If $K$ is a $(2,n)$--cable any two such tangles are orientation preserving homeomorphic.
\end{thm}

\begin{proof}
There is an essential torus $T$ in $V-N(K)$  splitting the manifold into two cable spaces $M_1$ and $M_2$ each containing one boundary torus of $V-N(K)$, say $\bdry N(K) \subset \bdry M_1$ and $\bdry V \subset \bdry M_2$. Since any given strong involution of $V-N(K)$ must extend to the strong involution of the solid torus $V$ under the trivial filling of $K$, $T$ necessarily intersects the fixed set of that involution.  Let $\iota$ be such a strong involution of $V-N(K)$. We may take $T$ to be invariant under $\iota$, \cite{meeksscott, jacorubinstein}.  Consequentially, $\iota\vert_{M_i}$ is a strong involution for each $i=1,2$.  Since $M_i$ is a cable space, any other strong involution is equivalent to $\iota\vert_{M_i}$  \cite[Proposition~2.8]{paoluzzi}.  Thus Proposition~\ref{prop:paoluzzi} implies there are at most four non-equivalent strong involutions on $V-N(K)$.  Since the quotients of $V-N(K)$ by these strong involutions correspond to the results of mutation of the quotient tangle $\tau = (S^2 \times I, t)$ of $V-N(K)$ by $\iota$ along the Conway sphere $S$ to which $T$ descends under $\iota$, we only need to observe that generically two pairs of these mutants are isotopic rel--$\bdry$ and all four are isotopic rel--$\bdry$ in special situations.

Since $T$ separates, $S$  separates the two boundary components of the tangle.  On each side of $S$ are tangles $\sigma_1$ and $\sigma_2$ in $S^2 \times I$ whose branched double covers are the manifolds $M_1$ and $M_2$.  By Theorem~\ref{thm:uniquetangles} Case (II), each $\sigma_1$ and $\sigma_2$ are the sum of a $4$--braid and a rational tangle.  As discussed above, mutations of $\tau$ along $S$ produce different tangles if the automorphism of $(S, S \cap t)$ does not extend to an automorphism of either $\sigma_1$ or $\sigma_2$.  Since $\sigma_2$ is the sum of two tangles along a disk $D$, two automorphisms of $(S, S \cap t)$ extend to maps of $\sigma_2$ that preserve the sides of $D$ while the other two extend to maps of $\sigma_2$ that swap the sides of $D$.  Because the summands of $\sigma_2$, a rational tangle and a product tangle, are each invariant themselves under mutation, the tangles resulting from the first two mutations (which include the identity) are isotopic rel--$\bdry$ as are the tangles resulting from the last two mutations as indicated in Figure~\ref{fig:cabletanglemutation}. Indeed the first two are isotopic rel--$\bdry$ to $\sigma_2$.  When the rational tangle summand is integral (so that $K$ is a $(2,n)$--cable), all four are isotopic rel--$\bdry$ to $\sigma_2$ as illustrated in Figure~\ref{fig:equivalentsigmamutants}.
\end{proof}

\begin{thm}[Hyperbolic Knots]\label{thm:hyperbolicRSR}
Let $K$ be a hyperbolic knot in the solid torus $V$ such that it has non-trivial surgery producing another solid torus. Then all strong involutions of $V-N(K)$ are isotopic.  

Hence any two tangles in $S^2 \times I$ whose branched double covers are $V-N(K)$ are orientation preserving homeomorphic.
\end{thm}

\begin{proof}
Since $V-N(K)$ is hyperbolic, Mostow-Prasad Rigidity \cite{mostow,prasad} and work of Waldhausen \cite{waldhausen} imply that a strong involution (or any automorphism) is isotopic to an isometry.  Quotienting $V-N(K)$ by this isometric strong involution produces a hyperbolic orbifold with underlying space $S^2 \times I$, ``pillowcase'' cusps, and order $2$ singular locus.  That is, the hyperbolic orbifold is a tangle in $S^2 \times I$ whose branched double cover is a hyperbolic manifold.  By the  Hyperbolic Dehn Surgery Theorem for Orbifolds \cite[Theorems 5.3 \& 5.4]{dunbar-meyerhoff} all but finitely many rational tangle fillings of a pillowcase cusp produce another hyperbolic orbifold in which the site of the rational tangle filling is a geodesic arc.  

The standard embedding of $V$ into $S^3$ has complementary solid torus $W^0$.  This embeds $K$ as $K^0$ in $S^3=V \cup W^0$.  Let $W^n$ be the solid torus resulting from $1/n$--surgery on the core of $W^0$ and let $K_n$ be the image of $K$ in $S^3=V \cup W^n$.

Assume $\iota_1$ and $\iota_2$ are two isometric strong involutions of $V-N(K)$ and their quotients are the hyperbolic orbifolds $\sigma_1$ and $\sigma_2$.  These strong involutions extend, at least topologically, across the fillings of $V-N(K)$ by $W^n$.  Let $\omega^n$ be the rational tangle filling of $\sigma_i$ corresponding to $W^n$.   Then, by the Hyperbolic Dehn Surgery Theorem for Orbifolds, for all but finitely many $n$ we have that $\sigma_i \cup \omega^n$ is a hyperbolic orbifold in which the site $a_i$ of the filling $\omega^n$ is a geodesic, for each $i=1,2$.

Take a suitably large $n$ so that $\tau_i=\sigma_i \cup \omega^n$ is a hyperbolic filling with geodesic site $a_i$ for both $i=1,2$.  Then the branched double cover of $\tau_i$ is the hyperbolic manifold $S^3-K^n = (V-K) \cup W^n$ in which $a_i$ lifts to a simple closed geodesic $A_i$.  The covering map gives an isometric strong involution $\widehat{\iota}_i$ of $S^3-K^n$ under which $A_i$ is strongly invertible.  Note that $\iota_i$ is isotopic to $\widehat{\iota}_i$ restricted to $V-N(K)$, the complement of a neighborhood of $A_i$.

Since  $A_1$ and $A_2$ are both geodesics and isotopic to the core of $W^n$ in $S^3-K^n$, $A_1=A_2$.  
Since $K^n$ is a strongly invertible hyperbolic knot in $S^3$ with a lens space surgery, it has only one strong inversion up to isotopy and no other symmetries \cite{wang-zhou}.  Thus $\widehat{\iota}_1 = \widehat{\iota}_2$ and hence $\iota_1$ is isotopic to $\iota_2$ on $V-N(K)$.
Moreover  $\tau_1=\tau_2$,  $a_1 = a_2$, and $\sigma_1=\sigma_2$.
\end{proof}

\subsection{Related conjectures}\label{sec:conj}

The above shows that the sites of an RSR between rational tangles up to homeomorphism are in 1-1 correspondence with the knots in a solid torus admitting a non-trivial Dehn surgery producing a solid torus.  What if two or more sites in a tangle correspond to the same knot?  While two such sites in a tangle may admit the same RSR producing tangles that have equivalent branched double covers, it seems unlikely that these resulting tangles should be too different.  The notion of ``same'' for the RSR should be regarded as being covered by the equivalent Dehn surgeries.

\begin{conj}\label{conj:twosites}
Let $\alpha_0$ and $\alpha_1$ be two sites for an RSR on a tangle $\tau$ that lift to isotopic knots in the branched double cover of $\tau$.  If the same RSR $\rho \mapsto \rho'$ on each site $\alpha_0$ and $\alpha_1$ produces tangles equivalent by a sequence of mutations then there must be a sequence of mutations of subtangles of $\tau$ taking $\alpha_0$ to $\alpha_1$.
\end{conj}

The condition that the lifts of $\alpha_0$ and $\alpha_1$ are isotopic is necessary for the conjecture to be possibly true.  For example, there are pairs of distinct Berge knots in $S^3$ for which the same distance $1$ surgery produces the same lens space, see the table in \cite{berge-lens}.  As these knots are strongly invertible, quotienting gives two sites in the tangle $\tau$, the ``closed tangle'' of the unknot in $S^3$, at which the same RSR produces the same $2$--bridge link $\tau'$.  (There is a unique link in $S^3$ that has a given lens space as its branched double cover \cite{hodgsonrubinstein}.)  Of course there is no homeomorphism of $\tau'$  relating the corresponding sites of the RSR because there is no homeomorphism of the lens space relating the dual knots of the two surgeries.  Otherwise the exteriors of these knots, and hence the exteriors of our original knots, would be homeomorphic contrary to knots in $S^3$ being determined by their complements \cite{gordonluecke} and our assumption that the original two knots are distinct.

\begin{remark}
By Wang-Zhou, a strongly invertible knot in $S^3$ with a non-trivial lens space surgery admits a single strong inversion, \cite{wang-zhou}.  Thus Conjecture~\ref{conj:twosites} is true when $\tau$ is the unknot in $S^3$ and the RSR produces a $2$--bridge knot or link.

The Berge knots are the conjecturally complete list of knots in $S^3$ admitting a non-trivial lens space surgery  (the Berge Conjecture \cite{berge-lens},  \cite[Problem 1.78]{kirby}).  Each has tunnel number $1$ and therefore admits a strong involution.  A positive resolution to the Berge Conjecture would then enable a complete classification of the sites of distance $d=1$ RSR between $2$--bridge links and the unknot.
\end{remark}

\begin{conj}\label{conj:1bridge}
If $K$ is a $1$--bridge braid in the solid torus $V$ then up to homeomorphism there is a unique tangle in $S^2 \times I$ whose branched double cover is the manifold $V-N(K)$.
\end{conj}

Every knot in a solid torus with a distance $1$ surgery producing a solid torus is a $1$--bridge braid \cite{gabai}.  
Notice that all $1$--bridge braids $K$ in $V$ have tunnel number $1$; that is, there is an arc $\alpha$ from $K$ to $\bdry V$ such that $V - N(K \cup \alpha)$ is a genus $2$ handlebody.  
\begin{lemma}
If $K$ is a tunnel number $1$ knot in a solid torus $V$, then $K$ is $1$--bridge.
\end{lemma}

\begin{proof}
Let $\tau_0$  be an unknotting tunnel for $K$.  That is, $\tau_0$ is an arc in $V$ with an endpoint on $K$ and an endpoint on $\bdry V$ such that $V-N(K \cup \tau_0)$ is a genus $2$ handlebody.  Attach a second solid torus $W$ with core $w$ to $V$ to form a lens space $M$.  Extend $\tau_0$ by a radial arc in $W$ to meet $w$ and use this to guide the placement of a narrow band $B$ (a rectangle) meeting $K$ and $w$ along opposite edges.  The other pair of opposite edges of $B$ union the arc $w-B$ is then an unknotting tunnel $\tau$ for $K$ in $M$.   Let $\gamma$ be the arc $B \cap K$.  Then by construction, $\tau \cup \gamma$ is isotopic to $w$ and thus its exterior is a solid torus.  Proposition~1.3 \cite{morimotosakuma} then implies that $\tau$ is what they call a $(1,1)$--tunnel for $K$ in $M$.  Therefore, as defined in \cite{morimotosakuma}, there is a dual tunnel $\tau^*$ for $K$ which may be regarded as being contained in $V$.  This dual tunnel confers a $1$--bridge presentation for $K$ in $V$.
\end{proof}

We now construct a family of tunnel number $1$ knots in the solid torus that are not braids yet admit two strong involutions whose quotients are distinct tangles in $S^2 \times I$.  These are offered to suggest Conjecture~\ref{conj:1bridge} is not true if its hypotheses are sufficiently weakened.  The following lemma will be useful in showing the two strong involutions are distinct.

\begin{lemma}
\label{lem:unknottingtunnelinvolution}
If $L$ is a tunnel number $1$ link of two components and $\alpha$ is an unknotting tunnel then the strong involution of $L$ arising from the hyperelliptic involution associated to the genus $2$ surface $\bdry N(L \cup \alpha)$ may be arranged to have $\alpha$ in its fixed set.
\end{lemma}

\begin{proof}
This follows directly from considering the strong involution of $L$ restricted to the handlebody $N(L\cup \alpha)$.  See also \cite{adams}.
\end{proof}

\begin{example}
\label{exa:twobridgelinkinv}

Let $L$ be the $2$--bridge link having corresponding continued fraction $[4a, 4b, 4a]$ with integers $a, b \neq 0$ as shown in Figure~\ref{fig:twobridgelink}.  Both components of $L$ are unknots, and since $L$ is a $2$--bridge link, it has tunnel number $1$.  We may then take the complement of one component to produce a solid torus $V$ in which the other component $K$ has tunnel number $1$.  One may show that $K$ is not a braid in $V$.  (Let $D$ be a meridional disk of $V$ that $K$ intersects algebraically and geometrically the same number of times. View $V-N(K)$ as a sutured manifold with $\bdry (V-N(K))$ as toroidal sutures.  Then one observes the sutured manifold $V'$ obtained by decomposing $V-N(K)$ along $D'=D\cap (V-N(K))$ is not a product sutured manifold by doing further decompositions along product disks.  If $K$ were a braid, then $D'$ would be a fiber in a fibration of $V-N(K)$, and $V'$ would be a product sutured manifold. See e.g.\ \cite{gabaiSM} or \cite{scharlSM}.)

%\begin{figure}
%\centering
%\includegraphics{twobridgelink.eps}
%\caption{}
%\label{fig:twobridgelink}
%\end{figure}

On the left of Figure~\ref{fig:twobridgelink} one finds a strong involution $\iota_0$ of $L$ rotating about a horizontal axis $F$, the fixed set of $\iota_0$.  We may then regard $\iota_0$ as an involution on the exterior of $L$  (i.e.\ $V-N(K)$)  whose quotient produces a $4$--strand tangle $\sigma_0$ in $S^2 \times I$ shown in the middle of Figure~\ref{fig:twobridgelink} (with Dehn surgery instructions).  If the involution $\iota_0$ were to arise from the hyperelliptic involution associated to an unknotting tunnel, then by Lemma~\ref{lem:unknottingtunnelinvolution} one of the four arcs of  $F-L$ must be an unknotting tunnel.  Only two of these arcs connect the two components of $L$ and they are related by another symmetry of $L$, so either both are unknotting tunnels or neither are.  Let $\alpha$ be one of these arcs.  Observe that because $a,b\neq0$, the tangle $L-N(\alpha)$ is a Montesinos  tangle and not a rational tangle (use the plane through $F$ orthogonal to the page to split the exterior of $\alpha$ into the two summands).  However, the classification of unknotting tunnels of $2$--bridge links \cite{adamsreid} shows that $L-N(\alpha)$ should be a rational tangle if $\alpha$ were to be an unknotting tunnel.  Hence $\iota_0$ is not a strong involution arising from an unknotting tunnel of $L$.  Therefore the strong involution $\iota_1$ on the exterior of $L$ arising from an unknotting tunnel (about the vertical axis in Figure~\ref{fig:twobridgelink}) has a quotient producing a different $4$--strand tangle $\sigma_1$ in $S^2 \times I$ shown on the right of Figure~\ref{fig:twobridgelink}.     

Indeed, the branched double covers of both $\sigma_0$ and $\sigma_1$ are homeomorphic.  Both $\bdry \sigma_0$ and $\bdry \sigma_1$ have a component that is the quotient of $\bdry N(K)$ under their corresponding involutions $\iota_0$ and $\iota_1$.  For evidence towards Conjecture~\ref{conj:twosites} observe that both have rational tangle fillings $\rho$ (the ``obvious'' closure of the bottom component) giving the same rational tangle $\tau = \sigma_0 \cup \rho=\sigma_1 \cup \rho$.  Moreover there is no homeomorphism of $\tau$ equating the sites of $\rho$ in these two fillings.  Furthermore, an RSR $\rho \mapsto \rho'$  will produce two tangles $\tau_0' = \sigma_0 \cup \rho'$ and  $\tau_1' = \sigma_1 \cup \rho'$ that have the same branched double cover.  Of course generically these tangles $\tau_0'$ and $\tau_1'$ will be inequivalent.  (For each individual case of $\rho'$, this may be seen by filling $\tau_0' = \sigma_0 \cup \rho'$ and $\tau_1'=\sigma_1 \cup \rho'$ with, say, a $+1$--tangle and observing the resulting links are inequivalent.)

Mecchia describes similar constructions of pairs of non-homeomorphic links (rather than tangles) in $S^3$ with homeomorphic branched double covers, \cite{mecchia}.
\end{example}

%%%%%%%%%%%%%%%%%%
\section{Knots in lens spaces with Seifert fibered exteriors.}\label{sec:seifertfibered}

Building on work of \cite{darcysumners} we catalog in Theorem~\ref{thm:seifertknots} all knots in lens spaces whose exteriors are generalized Seifert fibered spaces (rather than just ``true'' Seifert fibrations).  We further determine in Theorem~\ref{thm:seifertsurgery} all surgeries on such knots yielding lens spaces.   Let us highlight the identification and study of the knots isotopic to a regular fiber in a true Seifert fibration with non-orientable base  of a lens space, namely Lemmas~\ref{lem:GNO} though \ref{thm:kleinsurgery}.

Following the notation of \cite{JN} for Seifert fibered spaces, $M(g;(\alpha_1, \beta_1), \dots, (\alpha_n,\beta_n))$ denotes the generalized Seifert fibration over the closed surface of genus $|g|$ (a connect sum of $g$ tori if $g \geq 0$ and a connect sum of $-g$ projective planes if $g<0$) with the regular fibers running $\alpha_i$ times longitudinally and $\beta_i$ times meridionally about the $i$th exceptional fiber.

\begin{thm} \label{thm:seifertknots}
If $K$ is a knot in the lens space $Y$ such that its exterior $Y-N(K)$ admits a generalized Seifert fibration, then either $K$ is a torus knot or $K$ is a regular fiber of the true Seifert fibration $M(-1;(\alpha,1))$ of $L(4\alpha,2\alpha-1)$.
\end{thm}

Note this extends \cite{darcysumners} by considering generalized Seifert fibrations.

\begin{proof}
Assume $K$ is a knot in the lens space $Y$ such that its exterior $Y-N(K)$ admits a generalized Seifert fibration.  The Seifert fibration extends across $N(K)$ to give a generalized Seifert fibration of $Y$ in which $K$ is isotopic to a fiber.  By \cite{JN}, a generalized Seifert fibration of a lens space has the form $M(0;(\alpha_1,\beta_1),(\alpha_2,\beta_2))$ or $M(-1;(\alpha,1))$.  

If $K$ is isotopic to a regular fiber of a generalized and non-true Seifert fibration of $Y$ then it is isotopic to a meridional curve of an exceptional fiber.  Hence $K$ is a trivial knot (bounding a disk) and may be isotoped to be a meridional curve of a Heegaard solid torus and hence a torus knot. If $K$ is isotopic to a regular fiber of a true Seifert fibration with orientable base, then $K$ is a torus knot.  If $K$ is isotopic to a regular fiber of a true Seifert fibration with non-orientable base, then $K$ fits the second conclusion of the theorem.  If $K$ is isotopic to an exceptional fiber of a generalized (true or not) Seifert fibration with orientable base, then it is isotopic to the core of a Heegaard solid torus of a lens space and hence a torus knot.   If $K$ is isotopic to an exceptional fiber of a generalized (true or not) Seifert fibration with non-orientable base, then there is a non-fiber preserving homeomorphism $M(-1;(\alpha,1)) \cong M(0;(2,1),(2,-1),(-1,\alpha))$ (as noted in Theorem 5.1 \cite{JN}) in which $K$ is the ``exceptional'' fiber associated to $(-1,\alpha)$ (see the proof of Lemma~8 \cite{darcysumners} and the proof of Lemma~4 \cite{kohn}).  This Seifert fibration is equivalent to $M(0;(2,1),(2,2\alpha-1))$ in which $K$ is a regular fiber and hence isotopic to a torus knot as above.
\end{proof}

\begin{remark}
In the last case of Theorem~\ref{thm:seifertknots} $K$ is typically not a torus knot as we shall conclude in Lemma~\ref{lem:GNOtorusknot}.
\end{remark}

Express the lens space $L(4k,2k-1)$, $k>0$, as the union of two solid tori $V$ and $W$ along their boundary tori, and let $D_V$ and $D_W$ be meridional disks of these solid tori whose boundaries intersect (minimally) $4k$ times.  Around, say, $\bdry D_V$ number the intersections in order from $0$ to $4k-1$.
The {\em $2k$^{th} \ grid number one knot} in the lens space $L(4k,2k-1)$ is the union of a properly embedded arc in each $D_V$ and $D_W$ connecting the intersections $0$ and $2k$.  (This definition generalizes to speak of $p-1$ oriented grid number one knots (also called {\em simple} knots) in the lens space $L(p,q)$ \cite{bgh,hedden,rasmussen}, yet here we only need this specific case.)

\begin{lemma}\label{lem:GNO}
If $K$ is isotopic to a regular fiber of $M(-1;(k,1))$ for $k \neq 0$, it is the $2k$^{th} \ grid number one knot in $L(4k,2k-1)$.
\end{lemma}

\begin{proof}
If $K$ is isotopic to a regular fiber of $M(-1;(k,1))$ for $k \neq 0$, then $K$ is not necessarily a torus knot.   Decompose the Seifert fibration $M(-1;(k,1))$ as $(S^1 \tilde \times \mbox{\it Mob}) \cup (\mbox{\it Solid torus})$, the union of a twisted circle bundle over the \mobius band and a Seifert fibration of the solid torus with one exceptional fiber.      Observe that the manifold of this twisted circle bundle is homeomorphic to the exterior of a $(2,1)$--torus knot in $S^1 \times S^2$ (see \cite{kohn}) and also to a twisted $I$ bundle over the Klein bottle. 
Moreover, each regular fiber of $M(-1;(k, 1))$ is isotopic to a circle $S^1 \times \{{\rm pt}\}$ of this twisted circle bundle. 
 When expressed as $(I \tilde \times \mbox{\it Klein}) \cup (\mbox{\it Solid torus})$, a regular fiber of $M(-1;(k,1))$ may then be seen to be isotopic to a non-separating, orientation preserving curve on the Klein bottle.

Note that there are only four isotopy classes of essential simple closed curves on a Klein bottle.  Expressing a Klein bottle as the union of two \mobius bands, the cores of the two \mobius bands and their common boundary represent three of these isotopy classes.   The fourth is a union of a spanning arc of each \mobius band, and only it is non-separating and orientation preserving.

If our lens space $Y$ contains a Klein bottle, then $Y\cong L(4k, 2k-1)$ for some $k$ and it has only one up to isotopy, see e.g.\ \cite{bredonwood}.  It may be positioned as the union of two \mobius bands, one in each Heegaard solid torus.   Therefore, if $K$ is isotopic to a regular fiber of the Seifert fibration $M(-1;(k,1))$ of $Y$, then $Y$ contains a Klein bottle divided into two \mobius bands by the Heegaard torus, and $K$ is isotopic to a union of spanning arcs of these \mobius bands.   
This expression of $K$ as a union of spanning arcs of the two \mobius bands shows we may actually view $K$ as the $2k$^{th} \ grid number one knot in $Y$ when $k\neq0$.  
(When $k=0$, $K$ is a trivial knot in $Y \cong S^1 \times S^2$ and hence a torus knot.)
\end{proof}

\begin{lemma}\label{lem:fibered2k}
Let $K$ be the $2k$^{th} \ grid number one knot in the lens space $Y \cong L(4k,2k-1)$ for $k \neq 0$.  If $K$ is fibered, then $k=\pm1$.
\end{lemma}

\begin{proof}
 Since $K$ is a non-separating, orientation preserving curve on a Klein bottle $B$,  then $B \cut K$ is an annulus $A$.  Since an orientation on $A$  makes $\bdry A$ a pair of coherently oriented curves on $\bdry N(K)$, it follows that $A$ is a rational Seifert surface for $K$.  Since $K$ is not a core of a Heegaard solid torus, $A$ must be a minimal genus rational Seifert surface.  Assuming $K$ is fibered, then $A$ is a fiber and the monodromy of the fibration of $Y-N(K)$ is an orientation preserving homeomorphism of $A$ that swaps boundary components.  Hence $Y-N(K)$ is a twisted $I$ bundle over the Klein bottle, where each fiber of $K$ has the form $A = I \times \gamma$ and $\gamma$ is a non-separating, orientation preserving curve on the Klein bottle. (A Klein bottle may be fibered by such curves.)  
 
We may also view $Y-N(K)$ as a twisted $S^1$ bundle over the \mobius band where each annulus fiber is fibered by circles.  Then the torus $\bdry (Y-N(K))$ has a basis of the boundary of the \mobius band and a fiber. Since a meridian of $N(K)$ minimally intersects $\bdry A$ twice on $\bdry Y-N(K)$, it runs once in the direction of the boundary of the \mobius band and some number $\beta$ of times in the fiber direction.  Thus the twisted $S^1$ bundle over the \mobius band extends to a Seifert fibration $M(-1;(1,\beta))$ on $Y$.  Since $Y$ is a lens space, the classification of Seifert fibrations of lens spaces implies $\beta = \pm1$.  Hence $Y \cong L(4,1)$. Therefore $k=\pm1$.
 \end{proof}

\begin{lemma}\label{lem:GNOtorusknot}
Let $K$ be a regular fiber of the Seifert fibration $M(-1;(k,1))$ of a lens space $Y \cong L(4k, 2k-1)$.  
Then up to homeomorphism of the lens space, either
\begin{enumerate}
\item $k=0$ and $K$ is a trivial knot in $S^1 \times S^2$,
\item $|k| = 1$ and $K$ is a torus knot in $\pm L(4,1)$, or 
\item $|k| \geq 2$ and $K$ is toroidal, non-fibered, and not a torus knot.
\end{enumerate}
\end{lemma}

\begin{proof}
The case $k=0$ is discussed at the beginning of the proof of Theorem~\ref{thm:seifertknots}.  Lemma~\ref{lem:GNO} shows if $k \neq 0$, then $K$ is the  $2k$^{th} \ grid number one knot in $L(4k,2k-1)$.

If $k = \pm1$ then the grid diagram of $K$ gives an isotopy to a simple curve on the Heegaard torus as shown in the grid diagram for the second grid number one knot in $L(4,1)$ (Figure~\ref{fig:L41isotopy}).   (As such $K$ intersects a meridian of each Heegaard solid torus twice and hence bounds a \mobius band in each Heegaard solid torus.  Indeed, $K$ may be isotoped from its position on the Klein bottle in $\pm L(4,1)$ as an orientation preserving, non-separating curve to an orientation preserving, separating curve.)

If $|k|\geq 2$ then Lemma~\ref{lem:fibered2k} shows that $K$ is not fibered.  Because $K$ is non-null homologous, it is not  a trivial knot.  Since a Seifert fibered space over the disk with two exceptional fibers admits a fibration as a surface bundle over the circle, then all non-trivial torus knots in lens spaces are fibered.  Hence $K$ is not a torus knot.   Furthermore, since $K$ is an orientation preserving curve on a Klein bottle, it may be isotoped in the lens space to be disjoint from the Klein bottle.  After isotoping $K$ off into the solid torus complement $V$ of a regular neighborhood of the Klein bottle, the boundary of $V$ is then an essential torus in the exterior of $K$.  (It is incompressible into the neighborhood of the Klein bottle, and $K$ would be contained in a ball and hence null-homologous if it compressed into $V$.)
\end{proof}

%\begin{figure}
%\centering
%\includegraphics{L41isotopy.eps}
%\caption{A grid diagram for the $2$^{nd} grid number one knot in $L(4,1)$ with an isotopy of the knot onto the Heegaard torus.}
%\label{fig:L41isotopy}
%\end{figure}

\begin{thm}\label{thm:kleinsurgery}
Let $K$ be a regular fiber of the Seifert fibration $M(-1;(k,1))$ of a lens space $Y \cong L(4k, 2k-1)$.  Since $K$ is an orientation preserving curve on a Klein bottle, the Klein bottle endows $K$ with a framing.
If $K$ admits a non-trivial Dehn surgery yielding a lens space, then with respect to this framing either
\begin{enumerate}
\item $k=0$ and $\frac{1}{n}$--surgery on $K \subset S^1 \times S^2$ returns $S^1 \times S^2$ for all $n \in \Z$,
\item $k=\pm1 $ and $ \pm (1+\frac{1}{n})$--surgery on $K \subset \pm L(4,1)$ yields $ \pm L(4(n+1),2(n+1)-1)$, or
\item $k=\pm2$ and $\pm 1$--surgery on $K \subset \pm L(8,3)$ yields $\mp L(8,3)$.
\end{enumerate}
\end{thm}

\begin{proof}
Let $\mu$ and $\lambda$ be a standard meridian-longitude basis for $\bdry N(K)$ where $\lambda$ is the framing induced on $K$ by the Klein bottle.

If $k=0$, then $K$ is a trivial knot in $S^1 \times S^2$ and $\lambda$ is the standard framing.  Thus $\frac{1}{n}$--surgery on $K$ returns $S^1 \times S^2$.  Other surgeries would produce a connect sum of $S^1 \times S^2$ with a lens space other than $S^3$.

If $k=1$ then the isotopy of $K$ from its position as a non-separating, orientation preserving curve on the Klein bottle to its position as a torus knot shows its framing as a torus knot is $\lambda' = \lambda-\mu$.    With the framing $\lambda'$, $\frac{1}{n}$--surgery (i.e.\ surgery along the slope $n [\lambda']+[\mu]$) yields the lens space $L(4(n+1), 2(n+1)-1)$.   With respect to the framing $\lambda$ this surgery is $\frac{n+1}{n}=1+\frac{1}{n}$.   Mirroring provides the analogous statement for $k=-1$.

If $|k|\geq2$ then since $K$ may be isotoped to be disjoint from the Klein bottle and the complement of any Klein bottle in a lens space is a solid torus, any surgery on $K$ that returns another lens space must be a surgery that transforms this complementary solid torus into another solid torus.  In fact, since $K$ is a torus knot (winding more than once) in this complementary solid torus where $\lambda$ is still the framing, only $\frac{1}{n}$ surgeries potentially yield lens spaces.  

%\begin{figure}
%\centering
%\includegraphics{kleinbranch.eps}
%\caption{}
%\label{fig:kleinbranch}
%\end{figure}

We may view $Y \cong L(4k, 2k-1)$ as the branched double cover of the $2$--bridge link $S(4k,2k-1)$.  Indeed, $S(4k, 2k-1)$, shown in Figure~\ref{fig:kleinbranch}(a), is a pair of unknots each bounding a disk that the other geometrically intersects twice. In the branched double cover, each of these disks lifts to a Klein bottle (and the two lifted Klein bottles are isotopic).  An arc on one of these disks, as shown in Figure~\ref{fig:kleinbranch}(b), that connects the two intersections of the other unknot component lifts to a 
non-separating, orientation preserving curve on the Klein bottle lift of the disk.  Hence the lift of the arc is isotopic to our knot $K$ in $Y$.  Surgery on $K$ may then be viewed as the branched double cover of the link obtained by an RSR of a neighborhood of this arc.  The tangle corresponding to the exterior of $K$ is shown in Figure~\ref{fig:kleinbranch}(c).  The framing $\lambda$ on $K$ induced by the Klein bottle arises in the tangle picture from the disk that lifts to the Klein bottle.  Thus, since $|k|\geq2$, we only need to consider the links arising by twists shown in Figure~\ref{fig:kleinbranch}(d).   

%\begin{figure}
%\centering
%\includegraphics{kleinbranchk=2.eps}
%\caption{}
%\label{fig:kleinbranchk=2}
%\end{figure}

Because the result must be a $2$--bridge link, both components must be unknots.  As pictured, one component is a $(k,n)$--double twist knot.  In order for this component to be an unknot, we must have either $|k|\leq 1$ or $|n| \leq 1$.  Since $|k| \geq 2$ it must be that $|n| \leq 1$.  The case $n=0$ corresponds to the trivial surgery.  If $n = \pm 1$, then the component is a  $(2,k\mp1)$--torus link and hence is an unknot only if $k = \pm 2$.  Figure~\ref{fig:kleinbranchk=2} shows the situation $n=1, k=2$ where $S(8,3)$ is transformed into $S(8,5)=-S(8,3)$.  Thus, in the cover, when $k=2$, $+1$--surgery on $K \subset L(8,3)$ yields $-L(8,3)$.  Similarly, by mirroring, when $k=-2$, $-1$--surgery on $K \subset -L(8,3)$ yields $L(8,3)$.  Moreover when $|k|>2$, no non-trivial surgery on $K$ yields a lens space.
\end{proof}

\begin{remark}
Similar to $|k|\geq2$, the cases $|k| \leq 1$ in Theorem~\ref{thm:kleinsurgery} could have also been determined by examining the RSR of Figure~\ref{fig:kleinbranch}(c) that yield $2$--bridge links.  The RSR for $k=0$ and $k=1$ are shown in Figure~\ref{fig:kleinbranchk=01}.
%\begin{figure}
%\centering
%\includegraphics{kleinbranchk=01.eps}
%\caption{}
%\label{fig:kleinbranchk=01}
%\end{figure}
\end{remark}

\begin{remark}
Incidentally,  these distance $1$ RSR between the $2$--bridge links in Figures~\ref{fig:kleinbranchk=2} and \ref{fig:kleinbranchk=01} may be viewed as normal closures of distance $1$ RSR
between rational tangles.
\end{remark}

\begin{thm}\label{thm:seifertsurgery}
Let $K$ be a knot in a lens space $Y$ whose exterior admits a generalized Seifert fibration.  If a non-trivial surgery on $K$ yields a lens space, then one of the following occurs:
\begin{enumerate}
\item $K$ is the core of a Heegaard solid torus of the lens space, every surgery on $K$ yields another lens space, and every lens space may be obtained by surgery on $K$. 
\item $K$ is a $(p,q)$--torus knot in the lens space $L(r,s)$ and, using its framing as a torus knot, $\frac{1}{n}$--surgery yields the lens space $L(r+n\delta p,s+n\delta q)$ where $\delta =ps-rq$.  (If $|p|=1$ or $|ps-rq| = 1$ then $K$ is the core of a Heegaard solid torus as in the previous case and hence has more surgeries.)
\item $K$ is the $4$th grid number one knot in $\pm L(8,3)$ and, using its framing as an orientation preserving curve on a Klein bottle, $\pm1$--surgery on $K$ yields $\mp L(8,3)$.
\end{enumerate}
\end{thm}

\begin{proof}
Together Theorem~\ref{thm:seifertknots} and Lemma~\ref{lem:GNO} show that $K$ is either a torus knot or the $2k$^{th} \ grid number one knot in $\pm L(4k, 2k-1)$.  Surgeries on torus knots are well known and may be understood through Seifert fibrations,  see e.g.\ \cite{JN,moser,heil}.  
This gives the first two items.
Since trivial knots in lens spaces (such as $S^1 \times S^2$) are torus knots, Lemma~\ref{lem:GNOtorusknot} shows if $K$ is not a torus knot then $|k| \geq 2$.  Theorem~\ref{thm:kleinsurgery} then gives the third item.
\end{proof}

\begin{remark}
For non-integral surgeries, the results of the above theorem are obtained in \cite{darcysumners} through the Cyclic Surgery Theorem \cite{cgls} and the study of Seifert fibrations.
\end{remark}

%%%%%%%%%%%%%%%%%%%%%%%
\bibliographystyle{amsplain}
\bibliography{adjacent2.bib}
%%%%%%%%%%%%%%%%%%%%
%\newpage

\begin{figure}
\centering
\includegraphics{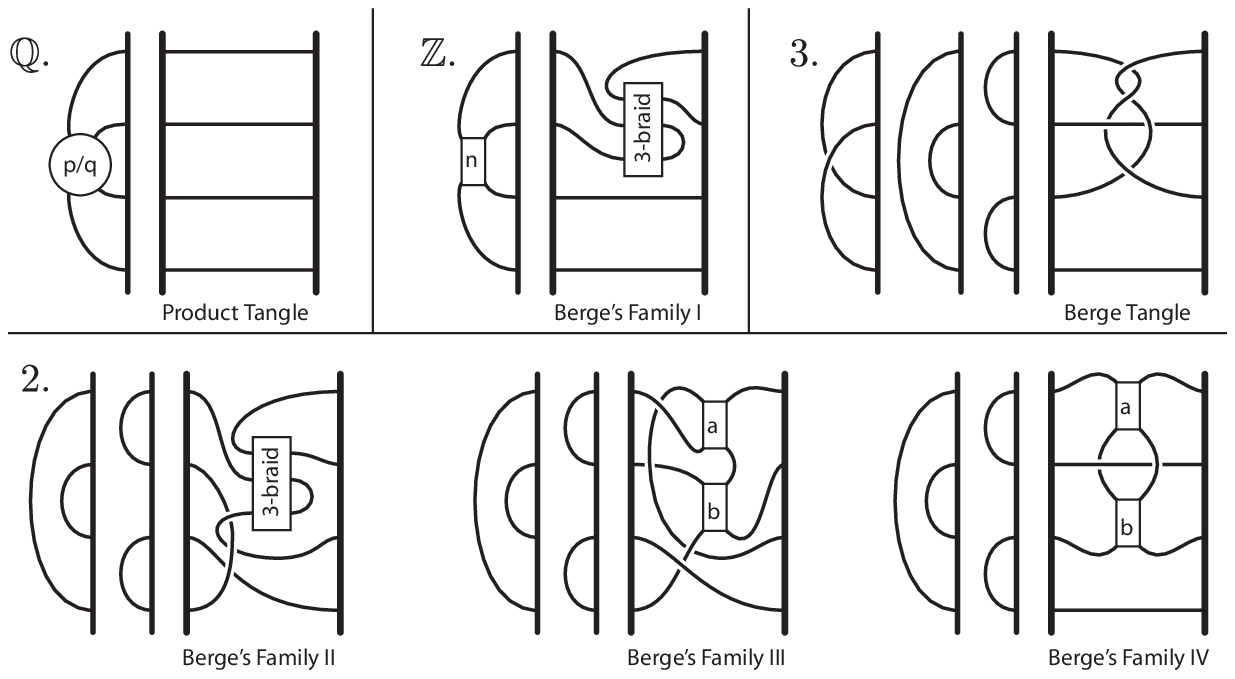}
\caption{}
\label{fig:fillingsofS2xI}
\end{figure}

\begin{figure}
\centering
\includegraphics{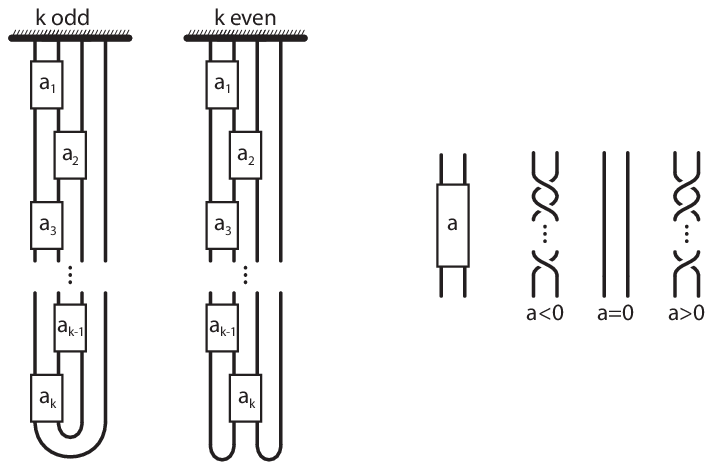}
\caption{}
\label{fig:rationalplat}
\end{figure}

\begin{figure}
\centering
\includegraphics{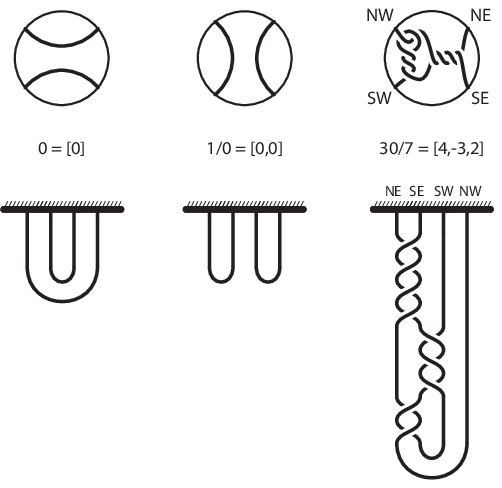}
\caption{}
\label{fig:darcysumnersfig}
\end{figure}

\begin{figure}
\centering
\includegraphics[angle=90, width=4.75in]{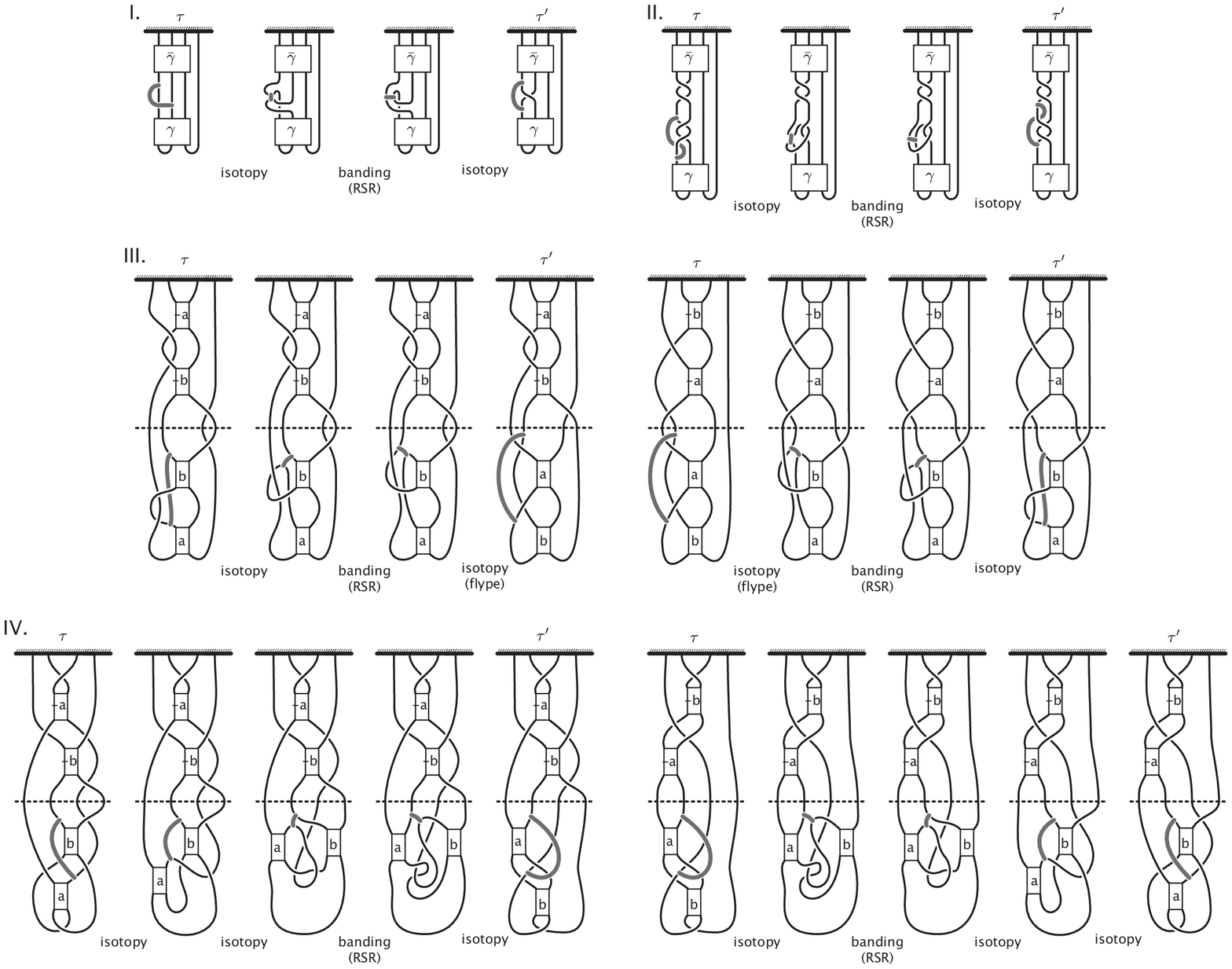}
\caption{}
\label{fig:d=1}
\end{figure}

%\begin{figure}
%\centering
%\includegraphics[angle=90,width=5in]{distanceonetanglesv2.eps}
%\caption{}
%\label{fig:d=1}
%\end{figure}

\begin{figure}
\centering
\includegraphics{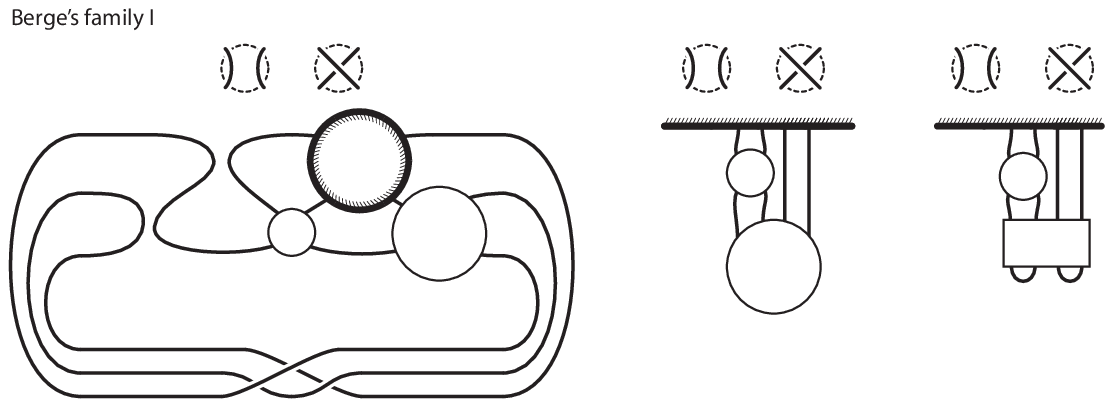}
\caption{}
\label{fig:bergeI}
\end{figure}

\begin{figure}
\centering
\includegraphics{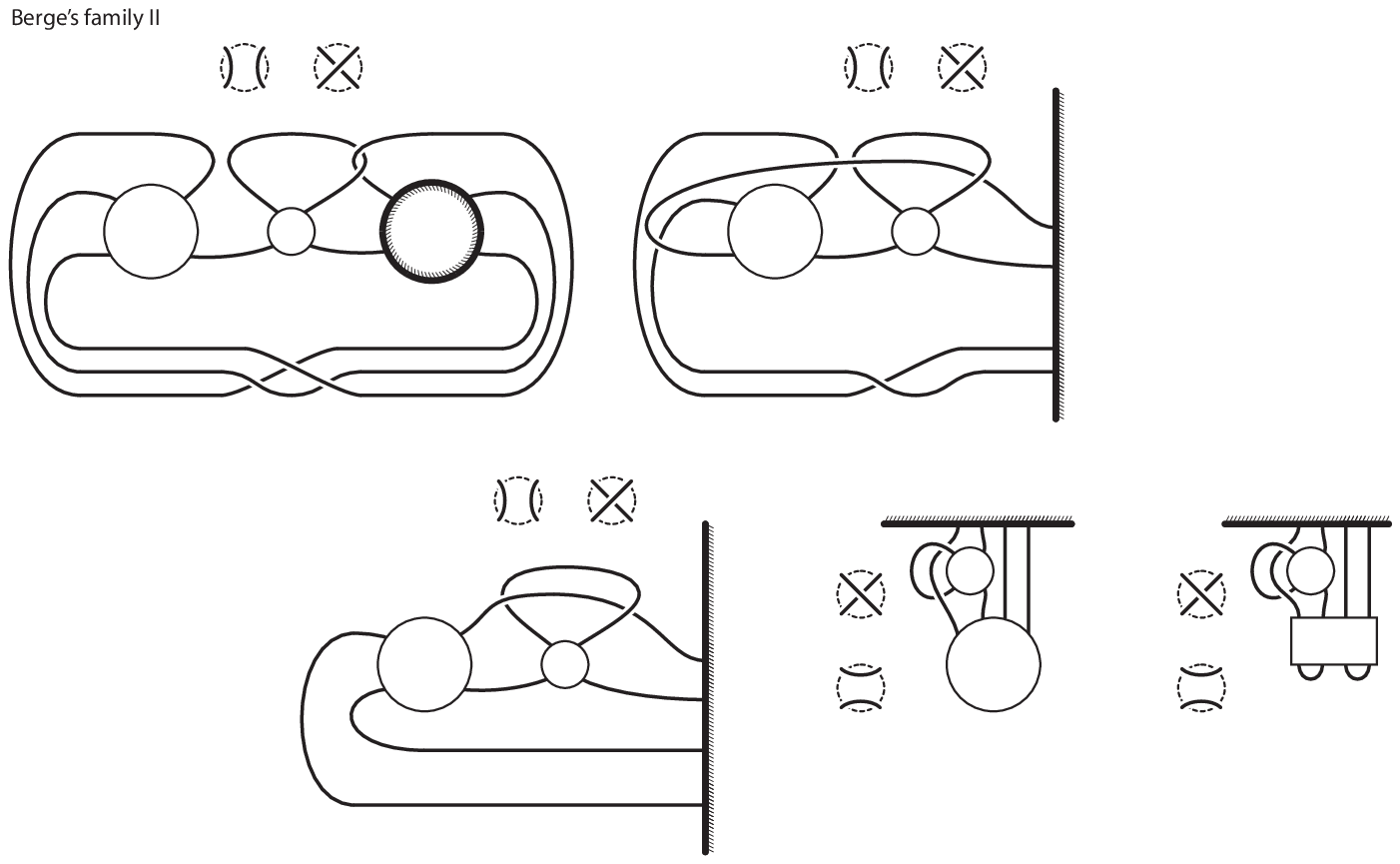}
\caption{}
\label{fig:bergeII}
\end{figure}

\begin{figure}
\centering
\includegraphics{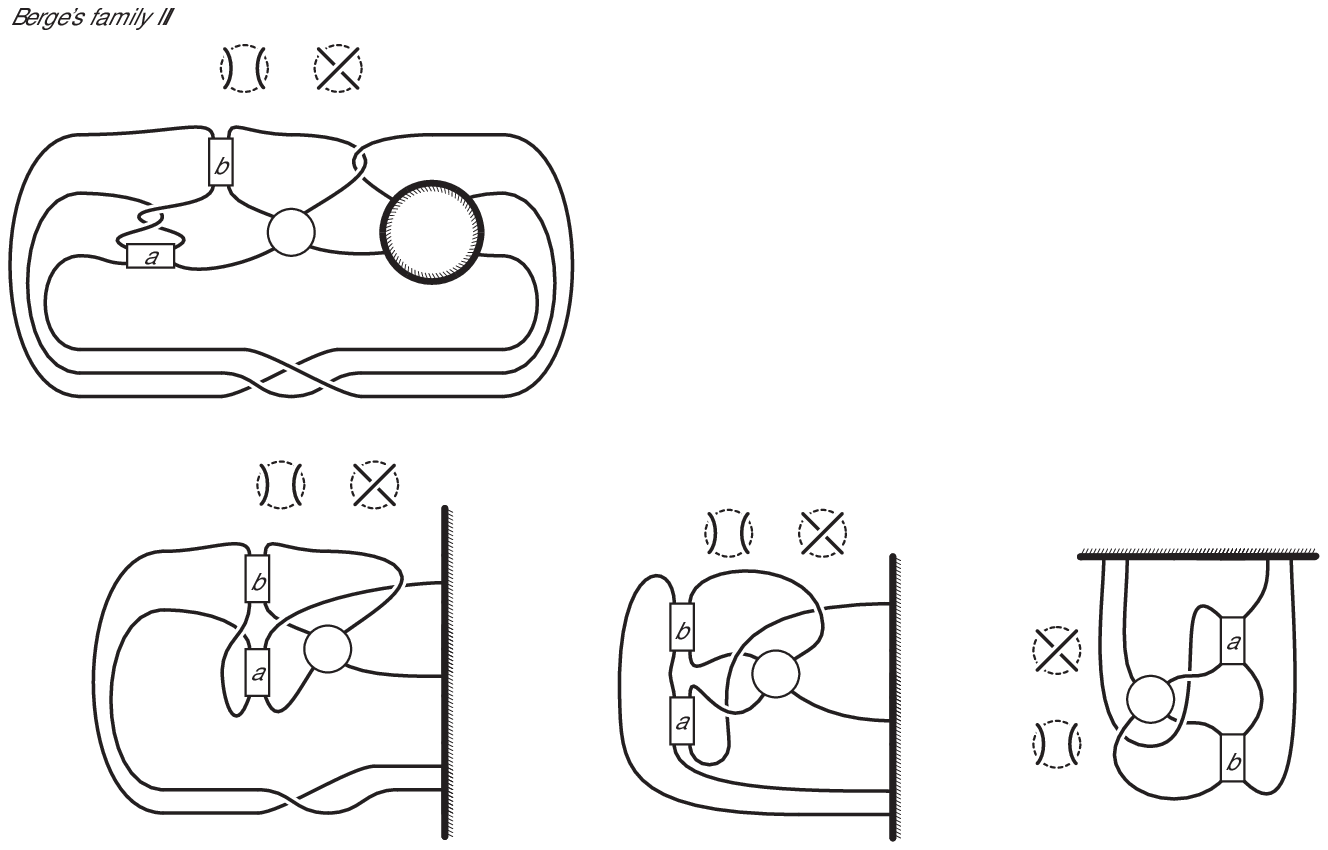}
\caption{}
\label{fig:bergeIII}
\end{figure}

\begin{figure}
\centering
\includegraphics{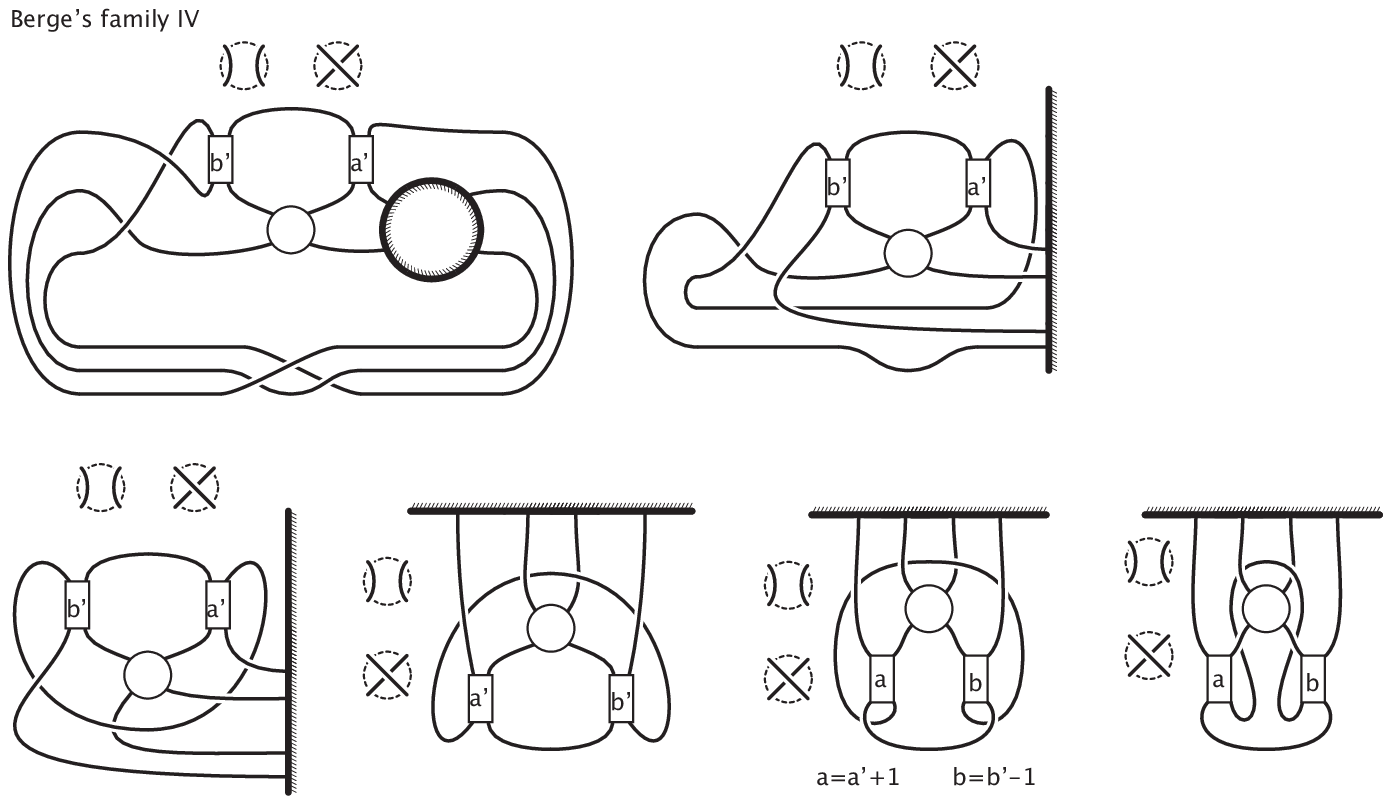}
\caption{}
\label{fig:bergeIV}
\end{figure}

\begin{figure}
\centering
\includegraphics{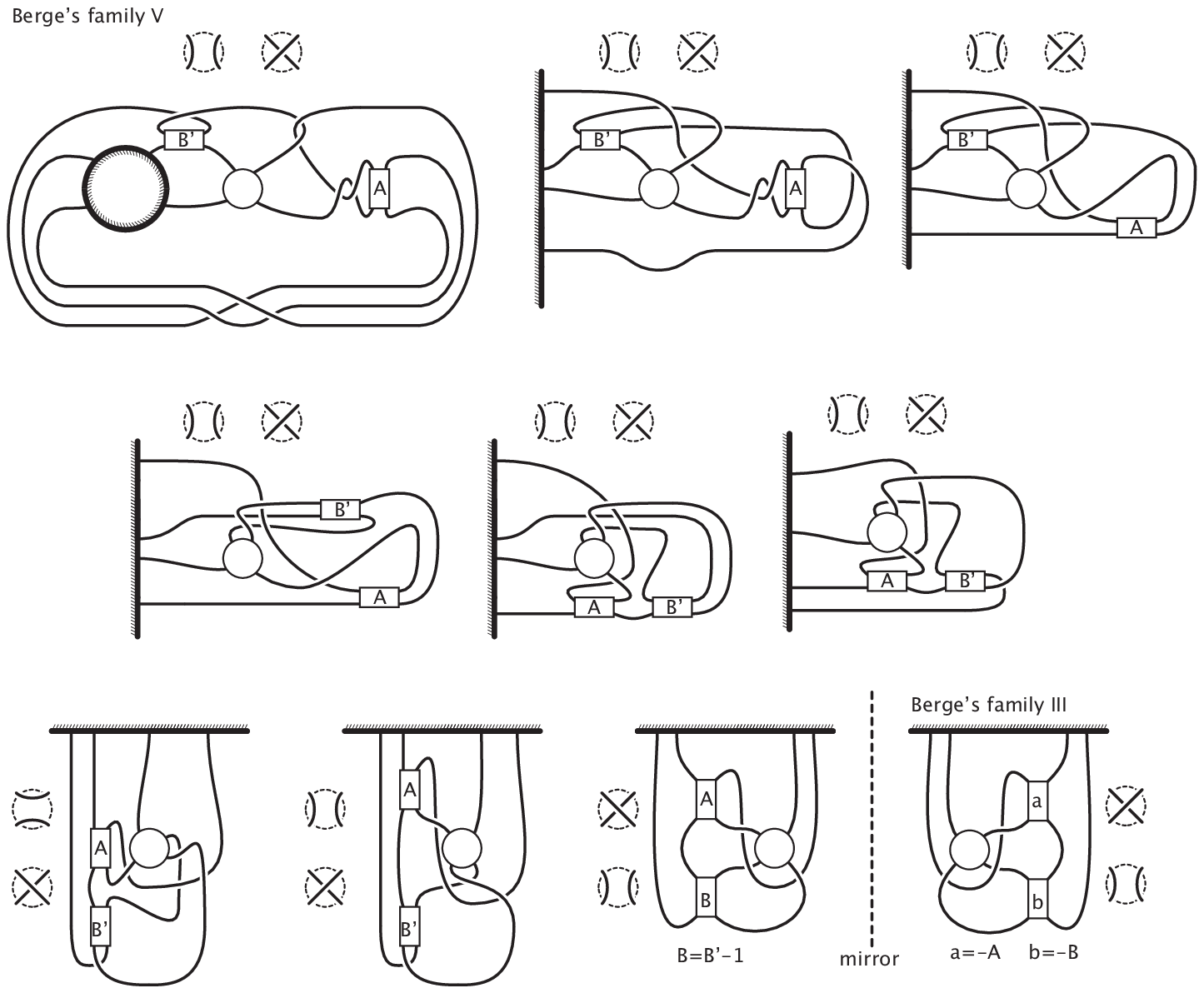}
\caption{}
\label{fig:bergeV}
\end{figure}

\begin{figure}
\centering
\includegraphics{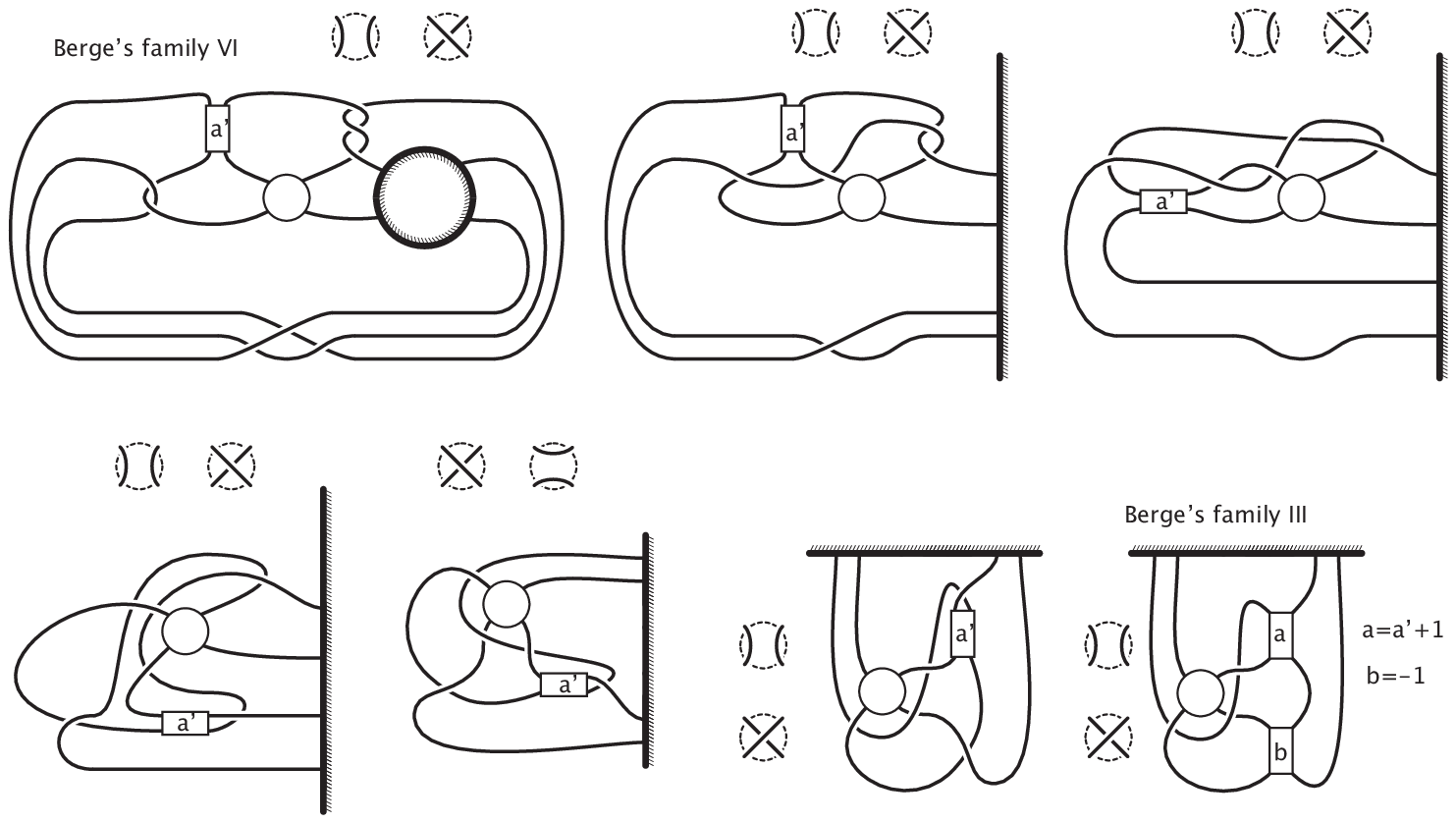}
\caption{}
\label{fig:bergeVI}
\end{figure}

\begin{figure}
\centering
\includegraphics{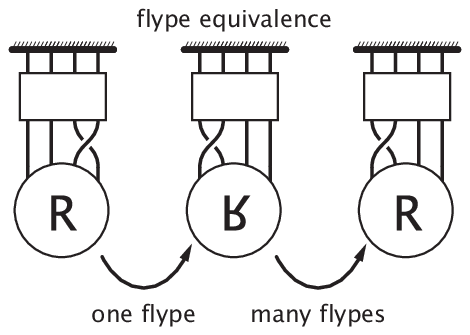}
\caption{}
\label{fig:flypeequivalence}
\end{figure}

\begin{figure}
\centering
\includegraphics{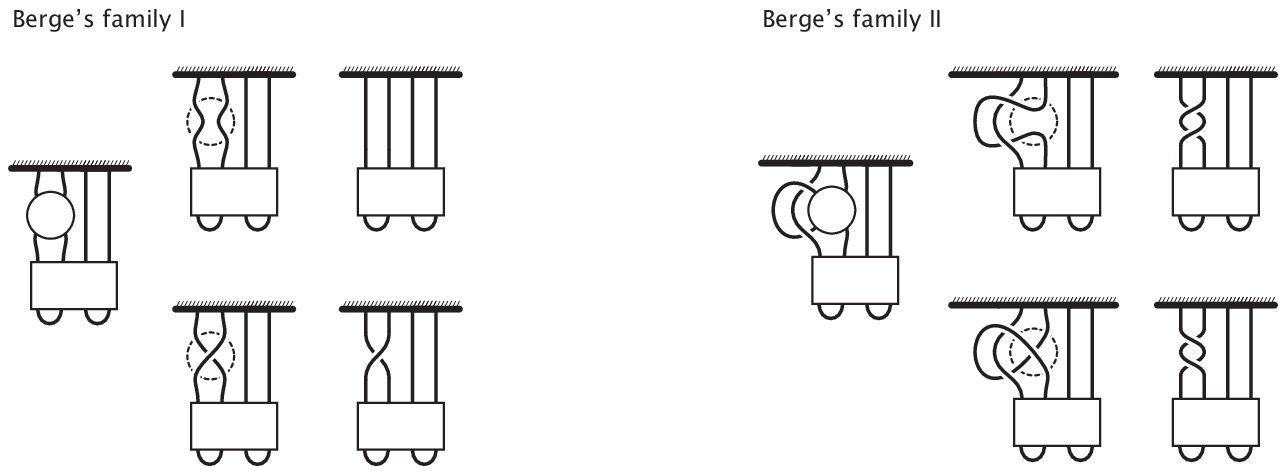}
\caption{}
\label{fig:bergechangeIandII}
\end{figure}

\begin{figure}
\centering
\includegraphics{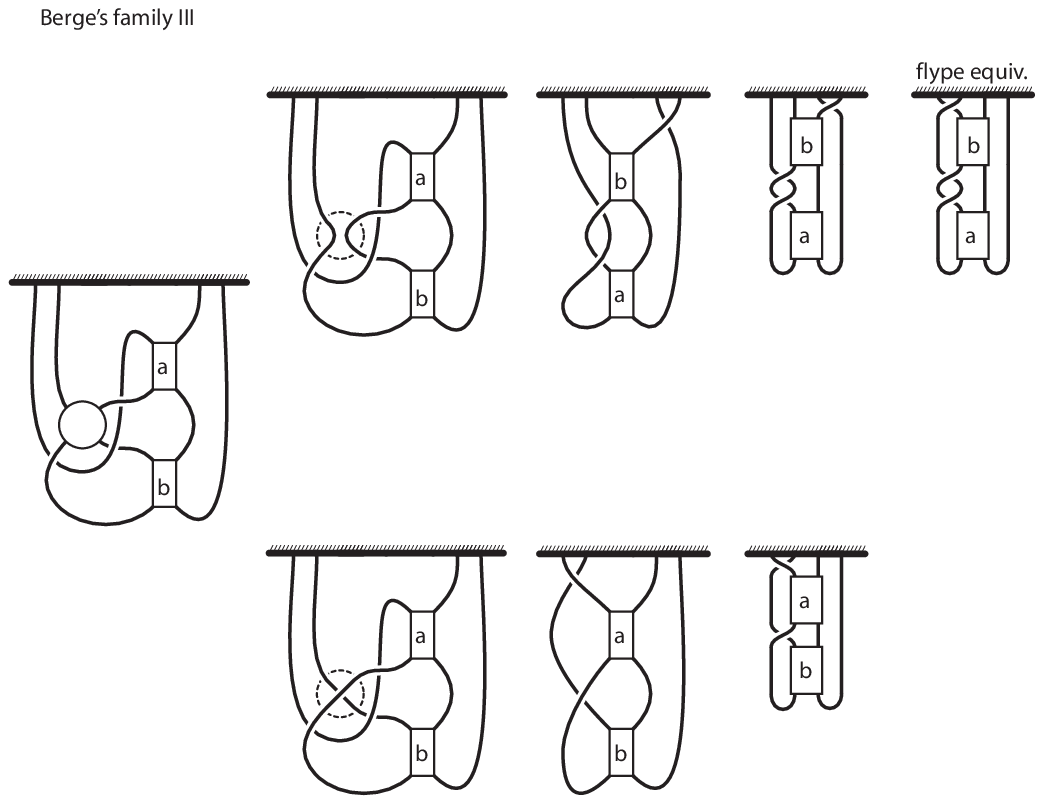}
\caption{}
\label{fig:bergechangeIII}
\end{figure}

\begin{figure}
\centering
\includegraphics{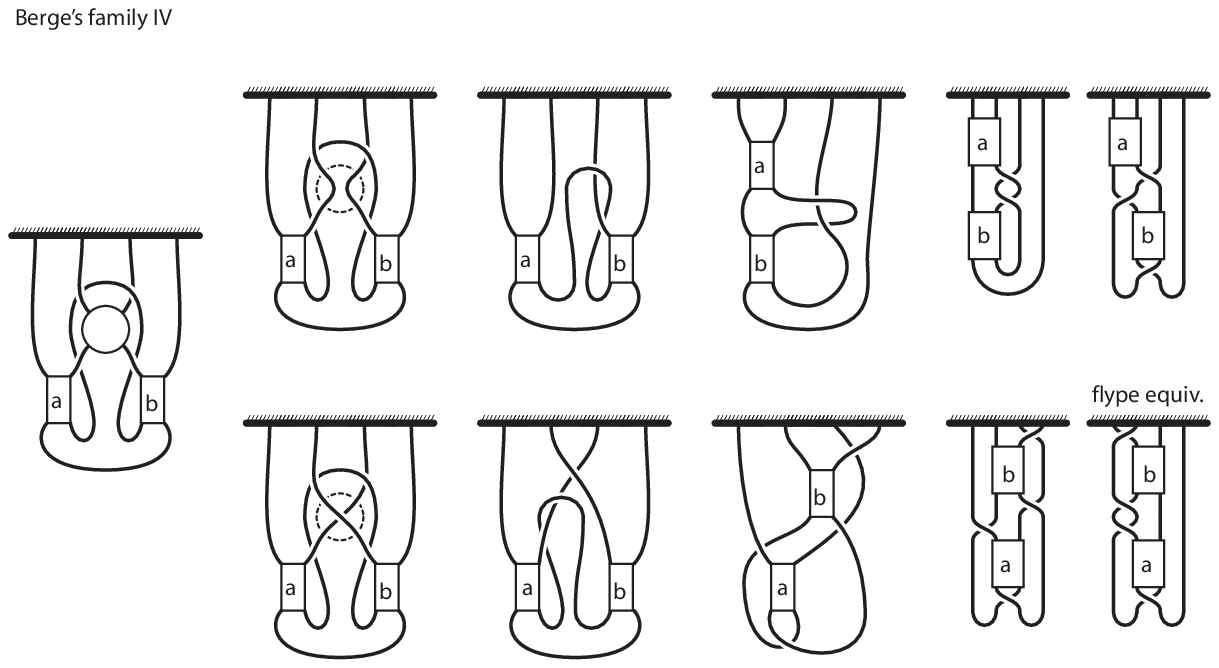}
\caption{}
\label{fig:bergechangeIV}
\end{figure}

\begin{figure}
\centering
\includegraphics{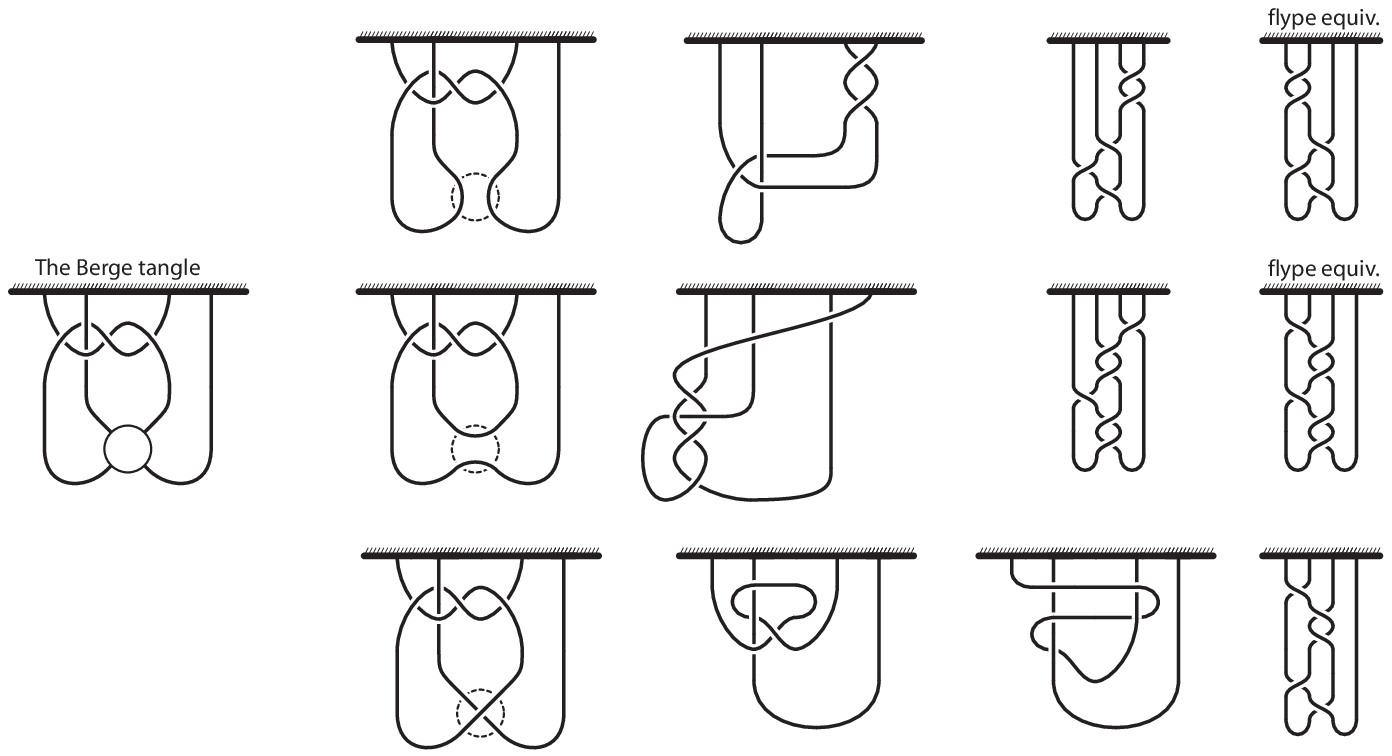}
\caption{}
\label{fig:BergeTangle}
\end{figure}

\begin{figure}
\centering
\includegraphics[height = 2in]{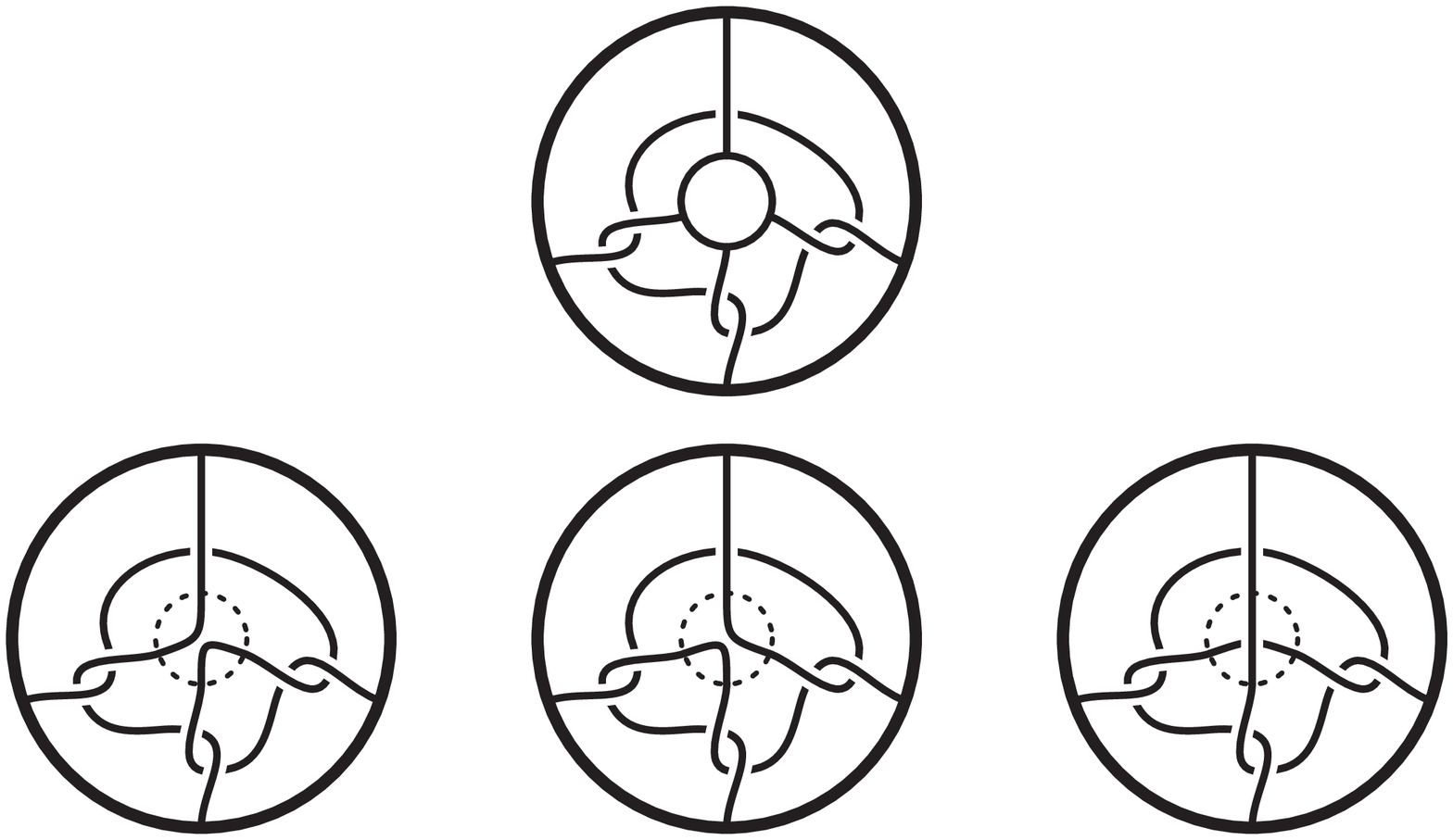}
\caption{}
\label{fig:tangleS1xS2}
\end{figure}

\begin{figure}
\centering
\includegraphics{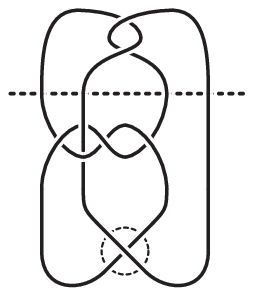}
\caption{}
\label{fig:8-17}
\end{figure}

\begin{figure}
\centering
\includegraphics{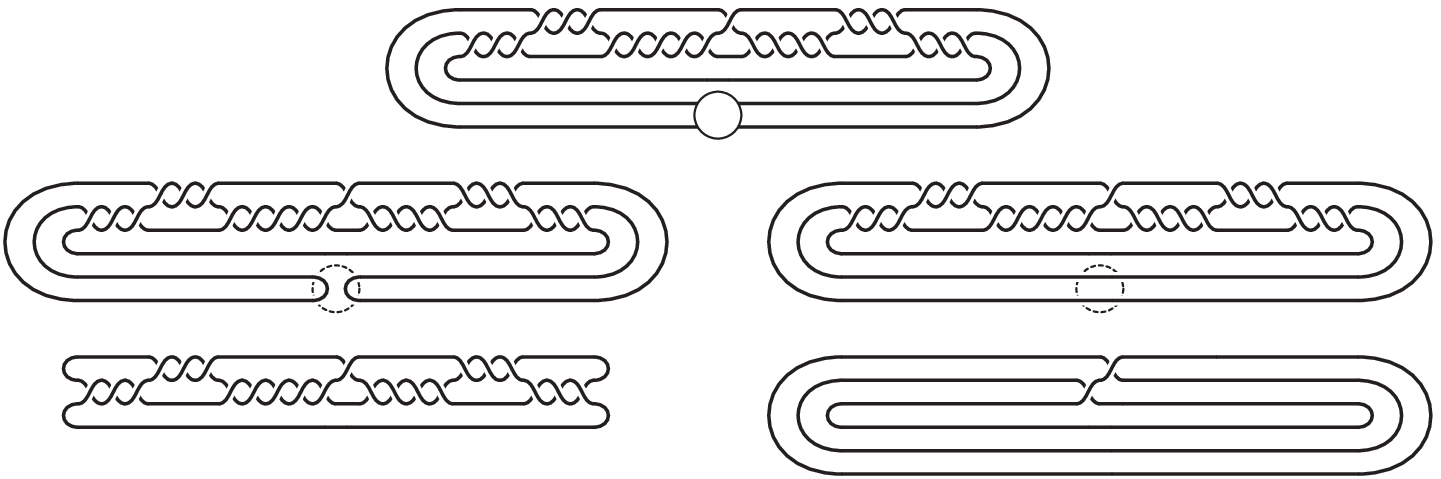}
\caption{}
\label{fig:largevoladjacency}
\end{figure}

\begin{figure}
\centering
\includegraphics[width=6in]{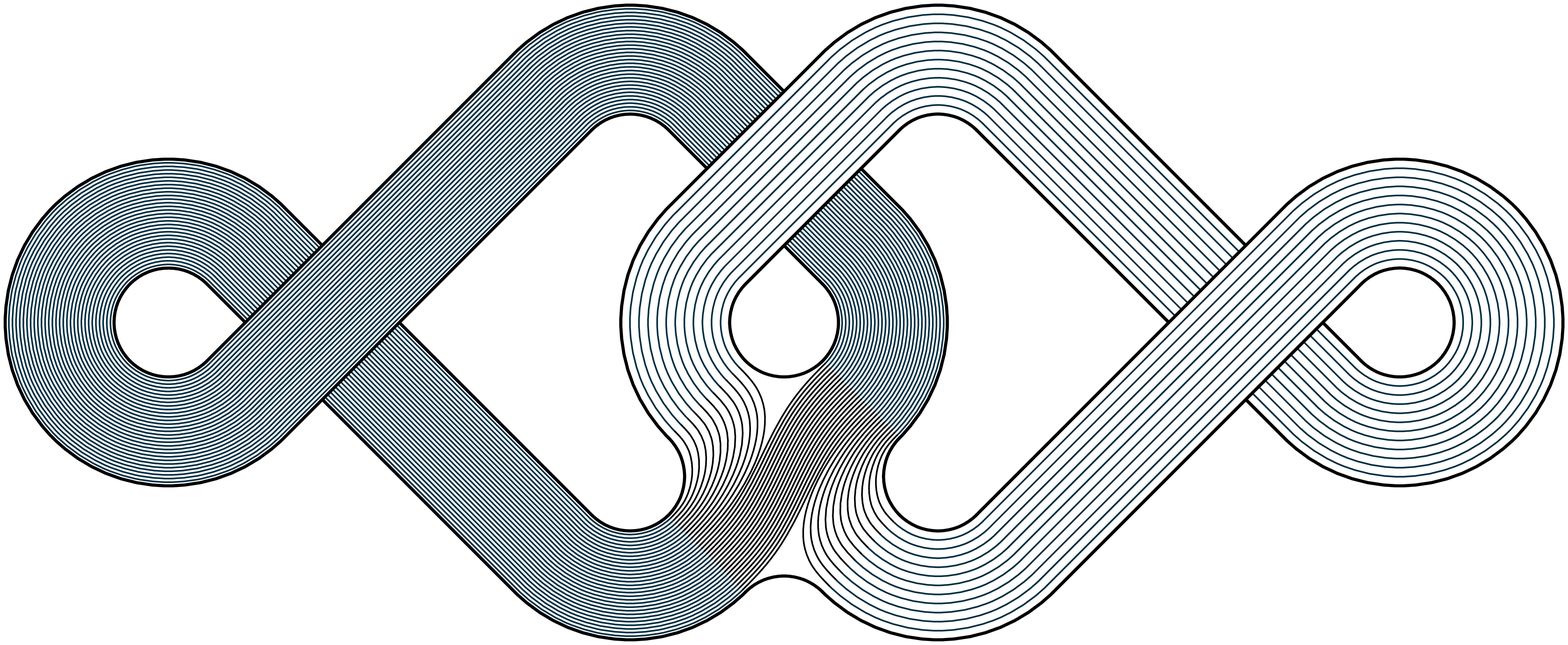}
\caption{}
\label{fig:fatknot}
\end{figure}

\begin{figure}
\centering
\includegraphics{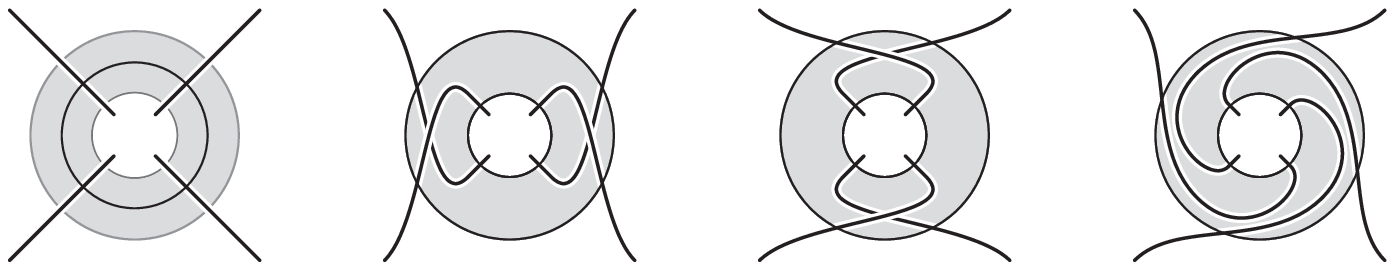}
\caption{}
\label{fig:mutation}
\end{figure}

\begin{figure}
\centering
\includegraphics{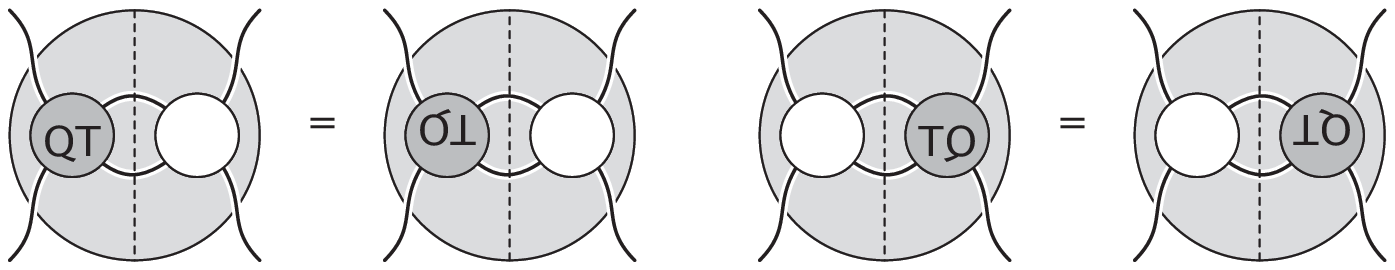}
\caption{}
\label{fig:cabletanglemutation}
\end{figure}

\begin{figure}
\centering
\includegraphics{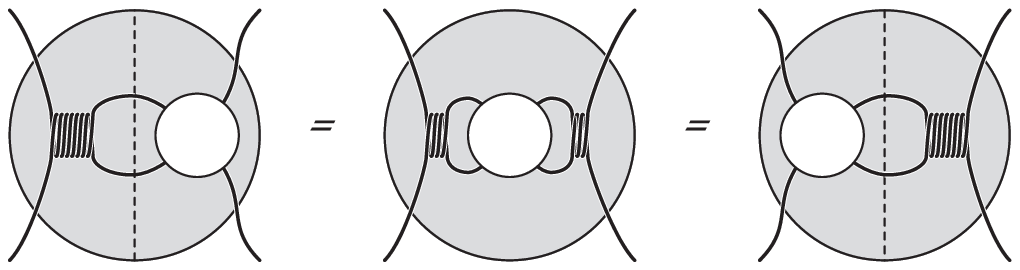}
\caption{}
\label{fig:equivalentsigmamutants}
\end{figure}

\begin{figure}
\centering
\includegraphics{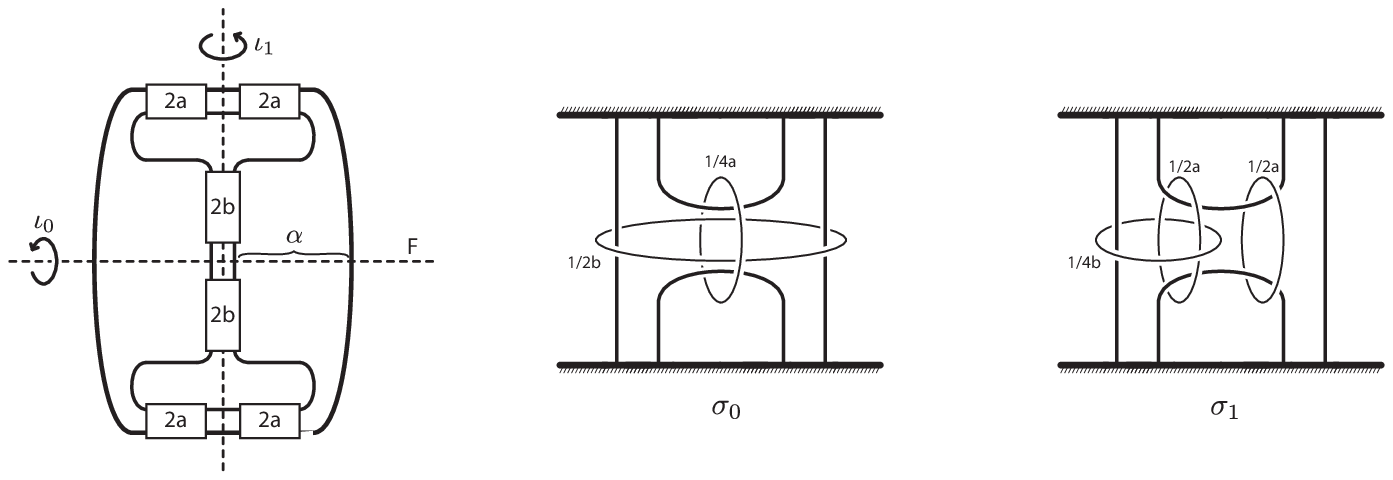}
\caption{}
\label{fig:twobridgelink}
\end{figure}

\begin{figure}
\centering
\includegraphics{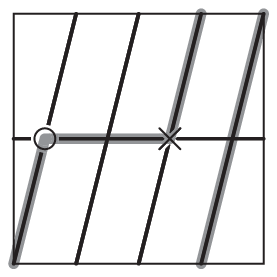}
\caption{}
\label{fig:L41isotopy}
\end{figure}

\begin{figure}
\centering
\includegraphics{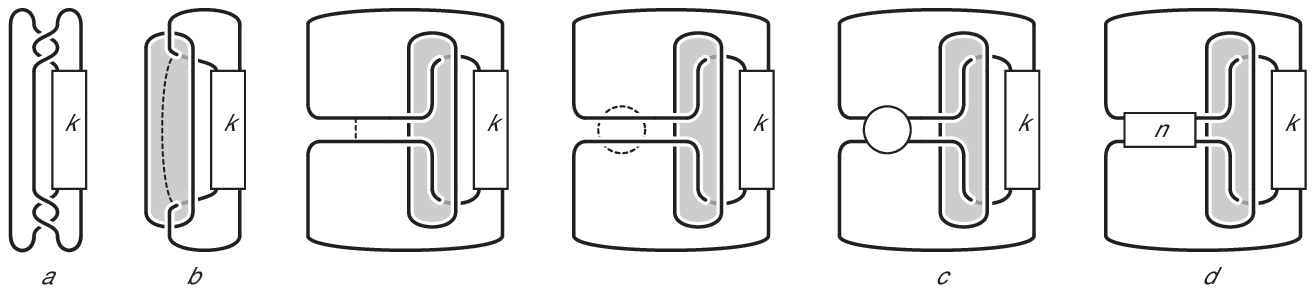}
\caption{}
\label{fig:kleinbranch}
\end{figure}

\begin{figure}
\centering
\includegraphics{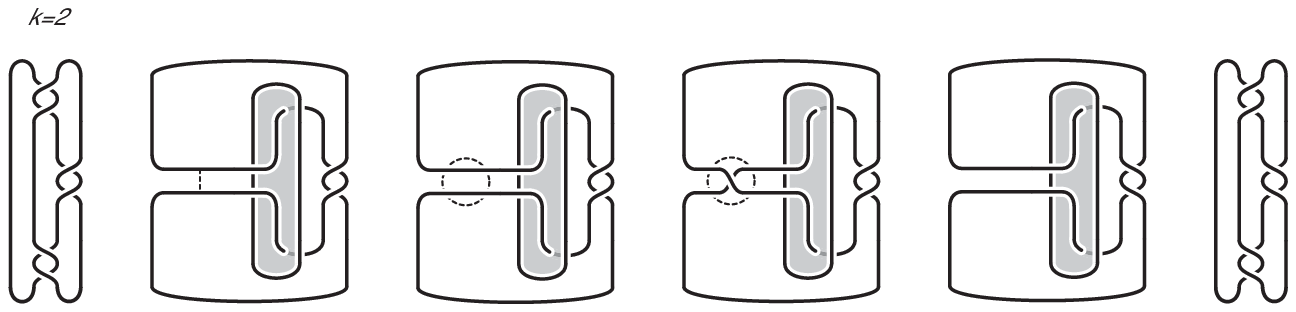}
\caption{}
\label{fig:kleinbranchk=2}
\end{figure}

\begin{figure}
\centering
\includegraphics{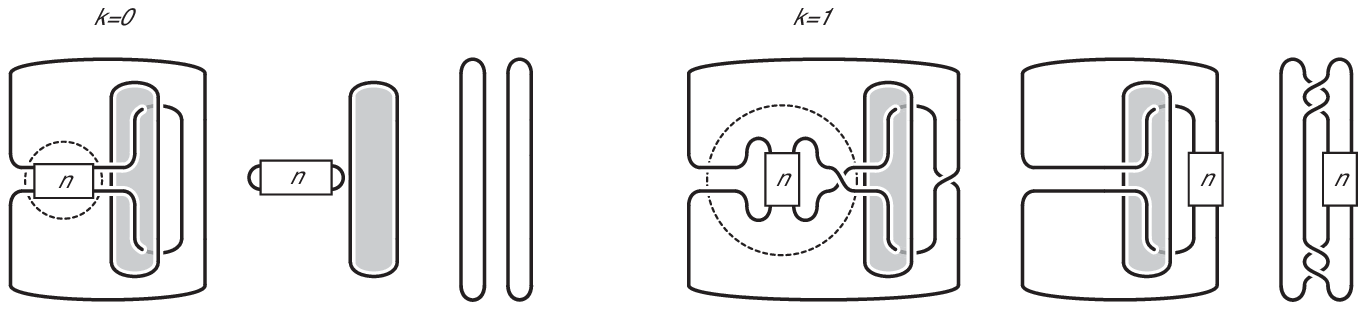}
\caption{}
\label{fig:kleinbranchk=01}
\end{figure}
%\end{remark}

\end{document}